\documentclass[11pt]{amsart}

\usepackage{epigamath2025}

\usepackage[english]{babel}

\numberwithin{equation}{section}

\usepackage{amsmath,latexsym,amsthm,cmap,indentfirst,setspace,amssymb,color,amscd,mathrsfs,euscript}

\usepackage[normalem]{ulem}
\usepackage{array}

\usepackage{lipsum}
\usepackage{float,xypic}
\usepackage{cmap,longtable,booktabs}
\usepackage[all,cmtip]{xy}

\usepackage[shortlabels]{enumitem}

\usepackage{mathtools}

\setlist[enumerate,1]{label=\rm{(\arabic*)}}
\setlist[enumerate,2]{label=\rm({\alph*})}
\newlist{a-enumerate}{enumerate}{2}
\setlist[a-enumerate,1]{label={\rm(\alph*)}}

\usepackage{tikz}
\usetikzlibrary{cd,arrows,positioning}
\tikzset{>=stealth}
\tikzcdset{arrow style=tikz}
\tikzset{link/.style={column sep=1.8cm,row sep=0.16cm}}
\tikzset{map/.style={row sep=0em, column sep=0em}}

\usetikzlibrary{cd,arrows,positioning,calc}
\usetikzlibrary{shapes.geometric}
\usetikzlibrary{intersections, through}

\theoremstyle{plain}
\newtheorem{theorem}{Theorem}[section]
\newtheorem{lemma}[theorem]{Lemma}
\newtheorem{corollary}[theorem]{Corollary}
\newtheorem{proposition}[theorem]{Proposition}
\newtheorem*{theorem*}{Theorem}

\theoremstyle{definition}
\newtheorem{definition}[theorem]{Definition}
\newtheorem*{notation}{Notation}

\theoremstyle{remark}
\newtheorem{remark}[theorem]{Remark}
\newtheorem{example}[theorem]{Example}
\newtheorem*{question*}{Question}

\DeclareMathOperator{\tr}{tr}
\DeclareMathOperator{\Bir}{Bir}
\DeclareMathOperator{\Spec}{Spec}
\DeclareMathOperator{\rk}{rk}
\DeclareMathOperator{\Pic}{Pic}
\DeclareMathOperator{\Aut}{Aut}
\DeclareMathOperator{\Ker}{Ker}
\DeclareMathOperator{\PGL}{PGL}
\DeclareMathOperator{\GL}{GL}
\DeclareMathOperator{\Br}{Br}
\DeclareMathOperator{\Gal}{Gal}
\DeclareMathOperator{\rank}{rank}
\DeclareMathOperator{\Center}{Z}
\DeclareMathOperator{\WW}{W}
\DeclareMathOperator{\Eu}{Eu}
\DeclareMathOperator{\GA}{AGL}

\newcommand{\I}{\ensuremath{\mathrm{I}}}
\newcommand{\II}{\ensuremath{\mathrm{II}}}
\newcommand{\III}{\ensuremath{\mathrm{III}}}
\newcommand{\IV}{\ensuremath{\mathrm{IV}}}

\newcommand{\CC}{\mathbb C}
\newcommand{\QQ}{\mathbb Q}
\newcommand{\FF}{\mathbb F}
\newcommand{\id}{\mathrm{id}}
\newcommand{\kk}{\mathbf k}
\newcommand{\PP}{\mathbb P}
\newcommand{\ZZ}{\mathbb Z}
\newcommand{\pt}{\mathrm{pt}}

\newcommand{\Sym}{\mathrm S}
\newcommand{\Alt}{\mathrm A}
\newcommand{\Dih}{\mathrm D}
\newcommand{\Cyc}{\mathrm C}
\newcommand{\Klein}{\mathrm V}

\DeclareMathOperator{\Image}{Im}

\def\lowsim{\vbox to 0pt{\vss\hbox{$\scriptstyle\sim$}\vskip-2pt}}
\newcommand\iso{\stackrel{\lowsim}{\to}}

\newcommand{\lra}{\longrightarrow}

\newcommand{\supth}[1]{\ensuremath{#1^{\mathrm{th}}}}
\newcommand{\suprd}[1]{\ensuremath{#1^{\mathrm{rd}}}}

\EpigaVolumeYear{9}{2025} \EpigaArticleNr{10} \ReceivedOn{July 25, 2023}
\InFinalFormOn{August 22, 2024}
\AcceptedOn{September 17, 2024}

\title{On $\boldsymbol{G}$-birational rigidity of del Pezzo surfaces}
\titlemark{$\boldsymbol{G}$-birational rigidity of del Pezzo surfaces}

\author{Egor Yasinsky}
\address{IMB, Universit\'{e} de Bordeaux, 351 Cours de la Lib\'{e}ration,
	33405 Talence Cedex, France}
\email{egor.yasinsky@u-bordeaux.fr}

\authormark{E.~Yasinsky}

\AbstractInEnglish{Let $G$ be a finite group and $H\subseteq G$ be a subgroup. We prove that if a smooth del Pezzo surface over an algebraically closed field is $H$-birationally rigid, then it is also $G$-birationally rigid, answering a geometric version of Koll\'{a}r's question in dimension 2  positively. On our way, we also investigate $G$-birational rigidity of two-dimensional quadrics and del Pezzo surfaces of degree 6.}

\MSCclass{14E07, 14E05, 14E30, 14J45, 14M22}
\KeyWords{Del Pezzo surface, conic bundle, birational rigidity, Sarkisov program}

\begin{document}



\maketitle

\begin{prelims}

\DisplayAbstractInEnglish

\bigskip

\DisplayKeyWords

\medskip

\DisplayMSCclass







\end{prelims}


\newpage

\setcounter{tocdepth}{1}

\tableofcontents


\section{Birational rigidity}\label{sec:intro}

The present note is motivated by J.~Koll\'{a}r's paper \cite{KollarRigidity} which studies the behaviour of birational rigidity of Fano varieties under extensions of algebraically closed fields. Let $\kk$ be a field. Recall that a \emph{Mori fibre space} is a projective morphism $\pi\colon X\to Y$ of algebraic varieties over $\kk$ such that $X$ is $\QQ$-factorial with terminal singularities, $\rk\Pic(X/Y)=1$ and $-K_X$ is $\pi$-ample. When $Y=\Spec(\kk)$, the variety $X$ is simply a $\QQ$-factorial terminal Fano variety of Picard rank~1. Roughly speaking, it is called \emph{birationally rigid} if $X$ is not birational to the total space of any other Mori fibre space; see Definition~\ref{def: BR} below for the details. In \cite{KollarRigidity} the following question was raised. 

\begin{question*}[Koll\'{a}r]
	Let $X$ be a Fano variety over a field $\kk$ such that $X$ is birationally rigid over the algebraic closure $\overline{\kk}$. Is $X$ birationally rigid over $\kk$?
\end{question*}

From the modern point of view, one can naturally formulate the minimal model program in the \emph{equivariant} setting; see \cite{ProkhorovGMMP} for an overview. Let $X$ be an algebraic variety over a field $\kk$ and $G$ be a group. Following Yu.~Manin \cite{ManinRationalI}, one calls $X$ a \emph{geometric $G$-variety} if $\kk=\overline{\kk}$ and there is an injective homomorphism $G\hookrightarrow\Aut(X)$. Another instance of this concept is an \emph{arithmetic $G$-variety}, for which $\kk\ne\overline{\kk}$ and $G$ is the Galois group of the extension $\overline{\kk}/\kk$ acting on $X\times_{\Spec(\kk)}\Spec(\overline{\kk})$ through the second factor; for simplicity, $\kk$ will be always assumed perfect in this paper. In both cases, we refer to $X$ as a \emph{$G$-variety} if no confusion arises.

\begin{remark}\label{rem: mixed case}
  Furthermore, one can consider the ``mixed'' case, when $\kk$ is not assumed to be algebraically closed and $G$ acts by biregular automorphisms of $X$; \textit{i.e.}~one 
  considers the action of $\Gal(\overline{\kk}/\kk)\times G$. For some results in this setting, see \cite{DolgachevIskovskikhPerfectField,TrepalinConicBundles,TrepalinDPhigh,TrepalinDelPezzo,RobayoZimmermann,YasinskyDelPezzo,CMYZ,AvilovFormsSegre} and Section~\ref{subsec: mixed} below. 
\end{remark}

Now, let us give some precise definitions. We follow \cite[Chapters 1--3]{CheltsovShramovIcosahedron}.

\begin{definition}\label{def: MFS}
	A \emph{$G$-Mori fibre space} is a $G$-equivariant surjective morphism $\pi\colon X\to Y$ of $G$-varieties such that $\pi$ has connected fibres, $X$ has terminal singularities, all $G$-invariant Weil divisors on $X$ are $\QQ$-Cartier divisors, $\dim Y<\dim X$, $\rk\Pic(X)^G-\rk\Pic(Y)^G=1$ and $-K_X$ is $\pi$-ample. If $Y$ is a point, we say that $X$ is a \emph{$G$-Fano variety}.
\end{definition}

\begin{definition}[see \textit{e.g.} {\cite[Definition 3.1.1]{CheltsovShramovIcosahedron}}]\label{def: BR}
	A $G$-Fano variety is called \emph{$G$-birationally rigid} if the following two conditions are satisfied:
	\begin{enumerate}
		\item\label{dBR-1} There is no $G$-birational map $X\dashrightarrow X'$ such that $X'$ is a $G$-Mori fibre space over a positive-dimensional variety.
		\item\label{dBR-2} If there is a $G$-birational map $\varphi\colon X\dashrightarrow X'$ such that $X'$ is a $G$-Fano variety, then $X\simeq X'$ and there is a $G$-birational self-map $\tau\in\Bir(X)$ such that $\varphi\circ\tau$ is a biregular $G$-morphism.
	\end{enumerate} 
	Assume, moreover, that the following holds:
	\begin{enumerate}[resume]
		\item\label{dBR-3} Every $G$-birational self-map $X\dashrightarrow X$ is actually $G$-biregular.  
	\end{enumerate}
	Then $X$ is called \emph{$G$-birationally superrigid}. 
\end{definition}

One can then generalize the initial question as follows. 

\begin{question*}[Cheltsov--Koll\'{a}r]
	Let $G$ be a group and $H\subseteq G$ be a subgroup. Assume that $X$ is an $H$-birationally rigid $H$-Fano variety. Is $X$ then $G$-birationally rigid? 
\end{question*}

Informally speaking, the non-triviality of this question lies \textit{e.g.} in the fact that \emph{a priori} $X$ may admit only $H$-birational maps to Fano varieties $X'$ with $\rk\Pic(X')^H>1$, while enlarging the group to $G$ forces $\Pic(X')^G$ to be of rank 1. Koll\'{a}r's original question addresses arithmetic $G$-varieties. The goal of this note is to answer its {\it geometric} counterpart in dimension~2 in the positive. 

\begin{theorem*}
	Let $\kk$ be an algebraically closed field of characteristic zero and $G$ be a finite group. Let $S$ be a two-dimensional geometric $G$-Fano variety over $\kk$, \textit{i.e.}~a smooth del Pezzo surface on which $G$ acts faithfully by automorphisms, so that $\Pic(S)^G\simeq\ZZ$. Assume that $H\subseteq G$ is a subgroup and $S$ is $H$-birationally rigid. Then $S$ is $G$-birationally rigid. 
\end{theorem*}

At the moment, there are no known counter-examples to Cheltsov--Koll\'{a}r's
question, neither in the geometric nor in the arithmetic setting, although it is highly probable such a counter-example exists. On the other hand, the results \cite{CheltsovShramovFiniteCollineationGroups} of I.~Cheltsov and C.~Shramov imply that the answer to the question is positive when $X=\PP_\CC^3$; see Section~\ref{sec: projective space} below. By contrast, a natural generalization of Cheltsov--Koll\'{a}r's question to the ``mixed'' setting of Remark~\ref{rem: mixed case} has a \emph{negative} answer. This will be shown in Section~\ref{subsec: mixed}.

Finally, we note that the analogue of Cheltsov--Koll\'{a}r's question with rigidity replaced by \emph{superrigidity} is easy to answer positively; \textit{i.e.}~one has the following. 

\begin{proposition}\label{prop: superrigidity}
	Let $G$ be a finite group and $H\subseteq G$ be a subgroup. Suppose that $X$ is an $H$-Fano variety. If\, $X$ is $H$-birationally superrigid, then $X$ is $G$-birationally superrigid.
\end{proposition}

\begin{proof}
	Recall, see \cite[Section~3.3]{CheltsovShramovIcosahedron}, that $X$ is $G$-birationally superrigid if and only if for every $G$-invariant mobile linear system $\mathscr{M}_X$ on $X$, the pair $(X,\lambda\mathscr{M}_X)$ is canonical for $\lambda\in\QQ_{>0}$ such that 
	\[
	\lambda\mathscr{M}_X\sim_\QQ -K_X.
	\]
	Since a $G$-invariant linear system $\mathscr{M}_X$ is also $H$-invariant, the result follows.
\end{proof}

\begin{remark}
	The analogue of Proposition~\ref{prop: superrigidity} holds when the notion of superrigidity is replaced with that of {\it solidity}. Recall, see \cite{AbbanOkadaPfaffian}, that a $G$-Fano variety $X$ is called {\it $G$-solid} if $X$ is not $G$-birational to a $G$-Mori fibre space with positive-dimensional base; \textit{i.e.}~only condition~\ref{dBR-1} of Definition~\ref{def: BR} is required; see \cite{CortiSingularitiesOfLinearSystems,CheltsovShramovIcosahedron,CheltsovShramovFiniteCollineationGroups,CheltsovSarikyan}. Recently, $G$-solid rational surfaces over the field of complex numbers were classified in \cite{Pinardin}.
\end{remark}

\begin{notation} We denote by $\Sym_n$ the symmetric group on $n$ letters and by $\Alt_n$ its alternating subgroup. Further, $\Dih_n$ denotes the dihedral group of order $2n$ with presentation
\[
\Dih_n=\left\langle r,s\ |\ r^n=s^2=(sr)^2=\id\right\rangle, 
\]
and $\Cyc_n$ denotes the cyclic group of order $n$. We denote by $\Klein_4$ the Klein four-group, isomorphic to $\Cyc_2\times\Cyc_2$. Finally, $A_\bullet B$ denotes an extension of $B$ by $A$, not necessarily split.
\end{notation}

\subsection*{Acknowledgements}
I would like to thank Ivan Cheltsov for bringing this problem to  my attention. I am very grateful to Andrey Trepalin for useful remarks and for pointing out some gaps in the first version of this text. Finally, I thank the anonymous referees for their helpful comments.

\section{Sarkisov links on del Pezzo surfaces}

\subsection{Geometric $\boldsymbol{G}$-surfaces}

In what follows, all surfaces are assumed smooth and projective and defined over an algebraically closed field $\kk$ of characteristic zero. Given a finite group $G$, a (geometric) {\it $G$-surface} is a triple $(S,G,\iota)$, where $S$ is a surface over $\kk$ and $\iota\colon G\hookrightarrow \Aut(S)$ is a faithful $G$-action. A $G$-{\it morphism} of $G$-surfaces $(S_1,G,\iota_1)\to (S_2,G,\iota_2)$ is a morphism $f\colon S_1\to S_2$ such that $f\circ\iota_1(G)=\iota_2(G)\circ f$. Similarly, one defines $G$-rational maps and $G$-birational maps. In what follows, we consider only rational $G$-surfaces equipped with the structure of a $G$-Mori fibre space, \textit{i.e.}~$G$-del Pezzo surfaces and $G$-conic bundles, whose $G$-invariant Picard ranks are 1 and 2, respectively.

\subsection{Sarkisov program}

The proof of the main theorem will be based on the explicit geometry of del Pezzo surfaces (for which we refer to \cite[Chapter 8]{DolgachevClassicalAG}) and, most importantly, the \emph{Sarkisov program} in dimension 2. Any $G$-birational map between two $G$-surfaces can be decomposed into a sequence of birational $G$-morphisms and their inverses. A birational $G$-morphism $S\to T$ is a blow-up of a union of $G$-orbits on $T$. In this article, we often refer to a $G$-orbit of size $d\geqslant 1$ as a \emph{$G$-point of degree $d$}. In particular, a $G$-point of degree 1 is simply a fixed point of $G$.

A $G$-birational map $\varphi$ between $G$-Mori fibre spaces $\pi\colon S\to B$ and $\pi'\colon S'\to B'$ is a diagram 
\[
\xymatrix{
	S\ar@{-->}[r]^{\varphi}\ar[d]_{\pi} & S'\ar[d]^{\pi'}\\
	B & B'
}
\]
which in general does not commute with the fibrations. Recall that in dimension 2, a rational $G$-Mori fibre space $\pi\colon S\to B$ is either a $G$-del Pezzo surface (if $B$ is a point) or a $G$-conic bundle (if $B=\PP^1$). According to the equivariant version of the Sarkisov program, every $G$-birational map $\varphi\colon S\dashrightarrow S'$ of $G$-Mori fibre spaces is factorized into a composition of isomorphisms of $G$-Mori fibre spaces and {\it elementary Sarkisov $G$-links} of four types, depicted below.
\[
\text{Type $\I$}
\xymatrix{
	&& T\ar@{->}[dll]_{\eta}\ar@{->}[d]^{\pi}\\
	S\ar[d]&& \PP^1\ar[dll]\\
	\pt &&
}\quad \text{Type $\II$}
\xymatrix{
	&T\ar@{->}[dl]_{\eta}\ar@{->}[dr]^{\eta'}&\\
	S\ar@{-->}[rr]^{\chi}\ar[dr]&& S'\ar[dl]\\
	& {B} &}
\]
\[
\text{Type $\III$}
\xymatrix{
	T\ar@{->}[drr]^{\eta}\ar@{->}[d]_{\pi} && \\
	\PP^1\ar[drr] && S\ar[d]\\
	&& \pt
}\quad \text{Type $\IV$}
\xymatrix{
	&T\ar@{->}[dl]_{\pi}\ar@{->}[dr]^{\pi'}&\\
	\PP^1\ar[dr]&& \PP^1\ar[dl]\\
	& \pt &}
\]
For type $\I$, $S$ is a $G$-del Pezzo surface and the Sarkisov link $\eta$ is the blow-up of a $G$-point on~$S$, giving a $G$-conic bundle $\pi\colon T\to\PP^1$; a link of type $\III$ is simply the inverse of $\I$. For type~$\II$, the birational morphisms $\eta$ and $\eta'$ are the blow-ups of $G$-points on~$S$ and $S'$, respectively. The induced Sarkisov link $\chi$ is either an elementary transformation of $G$-conic bundles (when $B\simeq\PP^1$) or a $G$-birational map between $G$-del Pezzo surfaces (when $B=\pt$). Finally, a link of type $\IV$ is the choice of a conic bundle structure on a $G$-conic bundle $T$ which has exactly two such structures; note that in general such a link is not represented by a biregular automorphism of~$T$, which exchanges $\pi$ and $\pi'$.

\begin{remark}\label{rem: rank r fibration}
	Recently, the Sarkisov program has been reformulated (and then successfully used to prove many structural results about Cremona groups) in terms of so-called \emph{rank~$r$ fibrations}; see \cite{LamyZimmermann,BLZ}. For example, in the arithmetic case, one defines a rank $r$ fibration as a surface $S$ with a surjective morphism $\pi\colon S\to B$ with connected fibres, where
	$B$ is a point or a smooth curve, with relative Picard number equal to~$r$ and $\pi$-ample anticanonical divisor~$-K_S$. Of course, in rank 1 fibrations, we recognize the usual (arithmetic) $G$-del Pezzo surfaces and $G$-conic bundles. The key observation, based on the so-called \emph{2-ray game}, is that rank 2 fibrations are in a one-to-one correspondence with Sarkisov links (and rank 3 fibrations correspond to the elementary relations between the links). As usual, there is a geometric counterpart of this theory; see \textit{e.g.} \cite{SchneiderZimmermann,Floris,FlorisZikas}. 
\end{remark}

Sarkisov links between surfaces were classified by V.\,A.~Iskovskikh in \cite[Theorem 2.6]{Isk1996} in the arithmetic case (see also a recent exposition \cite{LamySchneider} by S.~Lamy and J.~Schneider) and restated in \cite[Section 7]{DolgachevIskovskikh} in the geometric case. The following claim will be used systematically throughout the paper and is an immediate consequence of the Sarkisov program. 

\begin{proposition}
	A del Pezzo surface $S$ is $G$-birationally rigid if and only if for every Sarkisov $G$-link $S\dashrightarrow S'$, the surfaces $S$ and $S'$ are $G$-isomorphic.
\end{proposition}

\section{Del Pezzo surfaces of degree less than 6 and the projective plane}

In what follows, $S$ denotes a smooth del Pezzo surface over an algebraically closed field~$\kk$ of characteristic zero. Recall that $K_S^2\in\{1,2,\ldots,9\}$. Let $G\subseteq\Aut(S)$ be a finite group and $H\subsetneq G$ be a subgroup (we always stick to this notation in what follows). Assume that $\Pic(S)^H\simeq\ZZ$, so in particular $\Pic(S)^G\simeq\ZZ$. Note that the condition $\Pic(S)^H\simeq\ZZ$ immediately excludes two cases: first, when $K_S^2=7$, and second, when $K_S^2=8$ and $S$ is a blow-up of $\PP^2$ in one point. Indeed, in both cases, there exists an $H$-invariant $(-1)$-curve on $S$ which can be $H$-equivariantly contracted, so $\rk\Pic(S)^H>1$. 

\begin{proposition}\label{prop: Segre-Manin}
	Assume that $K_S^2\in\{1,2,3\}$ and $\Pic(S)^G\simeq\ZZ$. Then $S$ is $G$-birationally rigid.
\end{proposition}
\begin{proof}
	This is essentially the content of the so-called \emph{Segre--Manin theorem} (which follows from the classification of Sarkisov links nowadays). If $S\dashrightarrow S'$ is a $G$-birational map to another $G$-Mori fibre space $S'$, then it decomposes into Sarkisov $G$-links of type II and isomorphisms, and for any such link $\chi\colon S\dashrightarrow S'$, there exists a~commutative diagram
	\begin{equation}\label{eq: Sarkisov II}
	\xymatrix{
		&T\ar@{->}[dl]_{\eta}\ar@{->}[dr]^{\eta'}&\\
		S\ar@{-->}[rr]^{\chi}\ar[dr]&& S'\ar[dl]\\
		& \pt\rlap{,} &}
	\end{equation}
	where $\eta,\eta'$ are birational morphisms and $S',T$ are del Pezzo surfaces as well; see \cite[Propositions 7.12 and~7.13]{DolgachevIskovskikh}. This immediately implies that $S$ is even $G$-superrigid when $K_S^2=1$.
	If $K_S^2\in\{2,3\}$, then up to automorphisms of $S$, any such link is a~birational Bertini or Geiser involution; \textit{i.e.}~there exist a~biregular involution $\sigma\in\Aut(T)$ and an automorphism $\delta\in\Aut(S)$ such that
	$\chi=\eta'\circ\sigma\circ\eta^{-1}\circ\delta$.
	Thus, in particular, we have $S'\simeq S$. Since $\sigma$ centralizes $G$, we conclude that $S'$ is $G$-isomorphic to $S$.
\end{proof}

\begin{remark}
	In \cite{MauriDasDores}, M.~Mauri and L.~das Dores classify completely those del Pezzo surfaces of degree 2 and 3 which are $G$-birationally superrigid. 
\end{remark}

\begin{proposition}
	Let $K_S^2=4$. Then $S$ is $G$-birationally rigid if and only if there are no $G$-fixed points on $S$. In particular, if\, $S$ is $H$-birationally rigid, then it is $G$-birationally rigid as well. 
\end{proposition}
\begin{proof}
	By \cite[Propositions 7.12 and 7.13]{DolgachevIskovskikh}, every Sarkisov $G$-link $\chi$ starting from $S$ is either of type I or  of type II. In the former case, $\chi$ is centred at a $G$-fixed point on $S$. So, if such a point exists, then its blow-up is a smooth cubic surface equipped with a structure of $G$-conic bundle, so $S$ is not $G$-birationally rigid. Suppose there are no $G$-fixed points on~$S$. Then by \textit{loc.~cit.}, any Sarkisov $G$-link of type II starting from $S$ is either a birational Bertini involution (centred at a point of degree 3) or a birational Geiser involution (centred at a point of degree 2). In both cases, it leads to a $G$-isomorphic surface $S'\simeq S$, as in the proof of Proposition~\ref{prop: Segre-Manin}. Thus, $S$ is $G$-birationally rigid.   
\end{proof}

Before going to the next case, let us recall the following useful statement. 

\begin{lemma}[\textit{cf.} {\cite[Lemma 2.4]{BialynickiBirula}}]\label{lemma: fixed point}
	Let $X$ be an irreducible algebraic variety and $G\subset\Aut(X)$ be a finite group. If\, $G$ fixes a point $p\in X$, then there is a faithful linear representation $G\hookrightarrow\GL(T_pX)$. 
\end{lemma}

\begin{remark}\label{rem: fixed point cyclic}
	Note that a cyclic group always has a fixed point on a rational variety over an algebraically closed field of characteristic zero. This follows from the holomorphic Lefschetz fixed-point formula. 
\end{remark}

We now proceed with del Pezzo surfaces of degree 5.

\begin{lemma}[\textit{cf.}  {\cite[Theorem 6.4]{DolgachevIskovskikh}}]\label{lemma: dP5 minimal groups}
	Let $S$ be a del Pezzo surface of degree 5 and $G\subset\Aut(S)$ be a group such that $\Pic(S)^G\simeq\ZZ$. Then $G$ is isomorphic to one of the following five groups:
	\[
	\Sym_5,\quad\Alt_5,\quad \GA_1(\FF_5),\quad \Dih_{5},\quad \Cyc_5.
	\]
\end{lemma}

Here, $\GA_1(\FF_5)$ denotes the general affine group of degree 1 over $\FF_5$, defined by the presentation $\langle a,b\ |\ a^5=b^4=\id, bab^{-1}=a^3 \rangle $; it has the structure of a semidirect product $\Cyc_5\rtimes\Cyc_4$ and is sometimes called the Frobenius group of order 20.

\begin{proposition}[\textit{cf.} {\cite{WolterEquivariantBirationalDP5} and \cite[Example 6.3]{CheltsovLCThresholds}}]
	Let $K_S^2=5$. If\, $S$ is $H$-birationally rigid, then it is $G$-birationally rigid. 
\end{proposition}
\begin{proof}
	We use Lemma~\ref{lemma: dP5 minimal groups}. By \cite[Propositions 7.12 and~7.13]{DolgachevIskovskikh}, every Sarkisov $G$-link starting from $S$ is of type II; \textit{i.e.}~it is a diagram (\ref{eq: Sarkisov II}) where $\eta$ blows up a $G$-point of degree $d$, and one of the following holds:
	\begin{enumerate}
		\item\label{1} $S\simeq S'$, $d=4$, $\chi$ is a birational Bertini involution;
		\item\label{2} $S\simeq S'$, $d=3$, $\chi$ is a birational Geiser involution;
		\item\label{3} $S'\simeq\PP^1\times\PP^1$, $d=2$;
		\item\label{4} $S'\simeq\PP^2$, $d=1$.
	\end{enumerate}
	Recall that $\Alt_5$ (and hence $\Sym_5$) has no faithful representations of degree 2. Therefore, neither $\Alt_5$ nor $\Sym_5$ can have an orbit of size 1 or 2 on $S$ by Lemma~\ref{lemma: fixed point}. Further, $\Sym_5$ and $\Alt_5$ have no subgroups of index 3 or 4; hence there are no birational Bertini or Geiser involutions for these groups. We conclude that $S$ is $G$-birationally superrigid for $G\in\{\Alt_5,\Sym_5\}$. 
	
	So, it remains to verify the statement for the following pairs $(G,H)$:
	\[
	(\Cyc_5\rtimes\Cyc_4,\Dih_5),\quad (\Cyc_5\rtimes\Cyc_4,\Cyc_5),\quad (\Dih_5,\Cyc_5).
	\]
	Note that $S$ is never $H$-birationally rigid for $H\in\{\Cyc_5,\Dih_5\}$. Indeed, it is easy to show (using a holomorphic Lefschetz fixed-point formula, as was mentioned in Remark~\ref{rem: fixed point cyclic}) that $\Cyc_5$ has exactly {\it two} fixed points on $S$ and these points do not lie on $(-1)$-curves; see \textit{e.g.} \cite[Lemma 4.16]{YasinskyOdd}. Therefore, there is a Sarkisov link~\ref{4} from above, which leads to $\PP^2$. Furthermore, these two fixed points form an orbit under the action of the dihedral group $\Dih_5$, containing $\Cyc_5$. So, there is a link~\ref{3} leading to $\PP^1\times\PP^1$.  
\end{proof}

Let us treat the del Pezzo surface of degree 9, \textit{i.e.}~the projective plane $S=\PP^2$. We stick to the following classical (although perhaps outdated) terminology. 

\begin{definition}[\textit{cf.} {\cite{Blichfeldt}}]\label{def: Blichfeldt}
	We call a subgroup $\iota\colon G\hookrightarrow\GL_n(\kk)$ {\it intransitive} if the representation $\iota$ is reducible, and {\it transitive} otherwise. Further, a transitive group $G$ is called {\it imprimitive} if there is a decomposition $\kk^n=\bigoplus_{i=1}^mV_i$ into a direct sum of subspaces and $G$ transitively acts on the set $\{V_i\}$. A transitive group $G$ is called {\it primitive} if there is no such decomposition. Finally, we say that $G\subset\PGL_n(\kk)$ is (in)transitive or (im)primitive if its preimage in $\GL_n(\kk)$ is such a group.
\end{definition}

The following is due to D.~Sakovics. 

\begin{theorem}[\textit{cf.} {\cite[Theorem 1.3]{SakovicRigidity}}]\label{thm: Sakovics}
	The projective plane $\PP^2$ is $G$-birationally rigid if and only if\, $G$ is transitive and $G$ is not isomorphic to $\Alt_4$ or $\Sym_4$.
\end{theorem}

\begin{corollary}
	If\, $\PP^2$ is $H$-birationally rigid, then it is $G$-birationally rigid as well.
\end{corollary}
\begin{proof}
	Indeed, $G$ is transitive since $H$ is. Assume that $G\simeq\Alt_4$ or $G\simeq\Sym_4$. If $H\ne\Alt_4$, then $H$ must be one of the following groups: $\Cyc_2$, $\Cyc_3$, $\Cyc_4$, $\Cyc_2\times\Cyc_2$, $\Sym_3$, $\Dih_4$. But the irreducible representations of these groups are of degree 1 or 2; hence they fix a point on $\PP^2$ and are not transitive, so we have a contradiction (in fact, there is an $H$-equivariant blow-up $\FF_1\to\PP^2$). 
\end{proof}

\section{Del Pezzo surfaces of degree 6}\label{sec: dP6}

Let $S$ be a del Pezzo surface of degree $K_S^2=6$. Recall that $S$ is a blow-up $\pi \colon S\to \PP^2$ of three non-collinear points $p_1,p_2,p_3$, which we may assume to be $[1:0:0]$, $[0:1:0]$ and $[0:0:1]$, respectively. The set of $(-1)$-curves on $S$ consists of six curves: the exceptional divisors of blow-ups $e_i={\pi}^{-1}(p_i)$ and the strict transforms of the lines $d_{ij}$ passing through $p_i$ and $p_j$. In the anticanonical embedding $S\hookrightarrow\PP^6$, these exceptional curves form a ``regular hexagon'' $\Sigma$. This yields a homomomorphism to the symmetry group of this hexagon
\[
\psi\colon \Aut(S)\lra \Aut(\Sigma)\simeq\Dih_{6}=\left\langle r,s\ |\ r^6=s^2=1,\ srs=r^{-1}\right\rangle, 
\]
where $r$ is a rotation by $\pi/3$ and $s$ is a reflection, shown on Figure~\ref{pic:dP6}. The surface $S$ can be given as
\begin{equation}\label{eq: del Pezzo of degree 6}
\left\{([x_0:x_1:x_2],[y_0:y_1:y_2])\in\PP^2\times\PP^2:\ x_0y_0=x_1y_1=x_2y_2 \right\}. 
\end{equation}
The projection to the first factor $\PP^2$ is the blow-down of three lines $\{x_1=x_2=0\}$, $\{x_0=x_2=0\}$ and $\{x_0=x_1=0\}$ onto $p_1$, $p_2$ and $p_3$, respectively, while the projection to the second factor is the blow-down of $\{y_1=y_2=0\}$, $\{y_0=y_2=0\}$ and $\{y_0=y_1=0\}$.

The kernel of $\psi$ is the maximal torus $T$ of $\PGL_3(\kk)$, isomorphic to $(\kk^*)^3/\kk^*\simeq(\kk^*)^2$ and acting on $S$ by
\begin{equation}\label{eq: del Pezzo 6 torus action}
(\lambda_0,\lambda_1,\lambda_2)\cdot ([x_0:x_1:x_2],[y_0:y_1:y_2])=([\lambda_0 x_0:\lambda_1x_1:\lambda_2x_2],[\lambda_0^{-1}y_0:\lambda_1^{-1}y_1:\lambda_2^{-1}y_2])
\end{equation}
The corresponding element of $T$ will be denoted by $[(\lambda_1,\lambda_2,\lambda_3)]$. The action of $T$ on $S\setminus\Sigma$ is faithful and transitive. The automorphism group of $\Aut(S)$ fits into the short exact sequence 
\[
1\lra(\kk^*)^2\lra\Aut(S)\overset{\psi}{\lra}\Dih_6\lra 1 
\]
with $\psi(\Aut(S))\simeq\Dih_6\simeq\Sym_3\times\Cyc_2$. We denote by
\begin{equation}\label{eq: del Pezzo 6 Cremona}
\iota\colon([x_0:x_1:x_2],[y_0:y_1:y_2])\longmapsto ([y_0:y_1:y_2],[x_0:x_1:x_2])
\end{equation}
the lift of the standard Cremona involution, whose image under $\psi$ generates $\langle r^3\rangle\simeq\Cyc_2$, the centre of $\psi(\Aut(S))$. Further, the symmetric group $\Sym_3$ naturally acts on the indices of the coordinates $([x_0:x_1:x_2],$ $[y_0:y_1:y_2])$. In what follows, we will denote by 
\begin{equation}\label{eq: dP6 order 3 rotation}
	\theta\colon ([x_0:x_1:x_2],[y_0:y_1:y_2])\longmapsto ([x_1:x_2:x_0],[y_1:y_2:y_0])
\end{equation}
the automorphism which is mapped to $r^2$, the rotation of order 3 of the hexagon $\Sigma$.  
\begin{figure}[h]
	\centering
	\includegraphics[width=.3\linewidth]{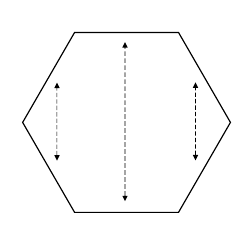}
	\caption{Action of $s$ on $\Sigma$}
	\label{pic:dP6}
\end{figure}

Note that 
\begin{equation}\label{eq: dP6 order 6 auto}
	\varrho=\iota\circ\theta\colon ([x_0:x_1:x_2],[y_0:y_1:y_2])\longmapsto ([y_1:y_2:y_0],[x_1:x_2:x_0]) 
\end{equation}
is an automorphism of order 6 such that $\psi(\varrho)$ generates $\langle r\rangle$. Finally, the automorphism
\begin{equation}\label{eq: dP6 sigma}
\sigma\colon ([x_0:x_1:x_2],[y_0:y_1:y_2])\longmapsto ([y_0:y_2:y_1],[x_0:x_2:x_1])
\end{equation}
is mapped onto the reflection of $\Sigma$. The automorphisms $\sigma$ and $\varrho$ generate a subgroup of $\Aut(S)$ which is mapped isomorphically onto $\Dih_6$ by $\psi$. In what follows, we sometimes call these actions $\theta$, $\iota$, $\varrho$ and $\sigma$ ``standard''.

\begin{lemma}\label{lem: dP6 fixed points}
	Let $S$ be a del Pezzo surface of degree 6 and $G\subset\Aut(S)$ be a finite group such that $\Pic(S)^G\simeq\ZZ$. If $G$ fixes a point on $S$, then $G\cap T=\id$.
\end{lemma}
\begin{proof}
	Assume that $G\cap T\ne\id$. Note that $T$ can be identified with a subgroup of $\PGL_3(\kk)$ which fixes three points $p_1$, $p_2$ and $p_3$. In particular, an element $t\in T$ fixing a point on $S\setminus\Sigma$ is necessarily trivial (of course, one can also deduce that from the explicit action of $T$ given above). Therefore, a fixed point $p\in S$ of $G$ lies on $\Sigma$. Note that it must be the intersection of two sides of $\Sigma$, as otherwise we have a $G$-invariant $(-1)$-curve, contradicting the minimality condition $\Pic(S)^G\simeq\ZZ$. Similarly, if $p=\ell_1\cap\ell_2$, where $\ell_1$, $\ell_2$ are some sides of $\Sigma$, then either $G$ preserves both $\ell_i$ and hence $\rk\Pic(S)^G>1$, or $G$ switches $\ell_1$ and $\ell_2$. Denoting by $\ell_1'$ and $\ell_2'$ the other two sides which intersect $\ell_1$ and $\ell_2$, respectively, we easily see that $\ell_1'$ and $\ell_2'$ form a $G$-orbit of non-intersecting $(-1)$-curves, so $\rk\Pic(S)^G>1$.
\end{proof}

The following elementary group-theoretic fact will be used several times below. We state it only for the reader's convenience.

\begin{lemma}\label{lem: dihedral group subgroups}
	The subgroups of the dihedral group $\Dih_n=\langle r,s\ |\ r^n=s^2=(sr)^2=\id\rangle$ are the following:
	\begin{description}
		\item[\rm Cyclic] $\langle r^d\rangle\simeq\Cyc_{n/d}$ and $\langle r^ks\rangle$, where $d$ divides $n$ and $0\leqslant k\leqslant n-1$; 
		\item[\rm Dihedral] $\langle r^d,r^ks\rangle\simeq\Dih_{n/d}$, where $d<n$ divides $n$ and $0\leqslant k\leqslant d-1$. 
	\end{description}
Moreover, all cyclic subgroups $\langle r^d\rangle$ are normal, one has $\Dih_n/\langle r^d\rangle\simeq\Dih_{d}$, and these are all normal subgroups when $n$ is odd. When $n$ is even, there are two more normal dihedral subgroups of index 2, namely $\langle r^2,s\rangle$ and $\langle r^2,rs\rangle$.
\end{lemma}

\begin{lemma}\label{lem: dP6 minimal groups}
	Let $S$ be a del Pezzo surface of degree 6. If\, $\Pic(S)^G\simeq\ZZ$, then $G$ is of the form
	\[
	N_\bullet\langle r\rangle\simeq N_\bullet \Cyc_6,\quad N_\bullet\langle r^2,s\rangle\simeq  N_\bullet\Sym_3\quad\text{or}\quad N_\bullet\langle r,s\rangle\simeq N_\bullet\Dih_6,
	\]
	where $N\simeq\Cyc_n\times\Cyc_m$ is a subgroup of\, $\Ker\psi\simeq (\kk^*)^2$.
        In particular, if\, $G$ fixes a point on $S$, then it is isomorphic to one of the following subgroups of\, $\Image\psi=\Aut(\Sigma)$: $\Cyc_6$, $\Sym_3$ or $\Dih_6$. 
\end{lemma}
\begin{proof}
	Let $G\subset\Aut(S)$ be a finite group. Assume that $\psi(G)$ is cyclic. In all the cases described in Lemma~\ref{lem: dihedral group subgroups} except $\psi(G)=\langle r\rangle$, we clearly have a $G$-orbit of skew sides of the hexagon, which correspond to $(-1)$-curves. Hence $\rk\Pic(S)^G>1$.
	
	Let $\psi(G)$ be dihedral, \textit{i.e.}~$\langle r,s\rangle$, $\langle r^2,s\rangle$, $\langle r^2,rs\rangle$, $\langle r^3,s\rangle$, $\langle r^3,rs\rangle$ or $\langle r^3,r^2s\rangle$. In the last three cases, one can always find a $G$-orbit of two disjoint $(-1)$-curves (a pair of opposite sides of the hexagon). Similarly, if $\psi(G)=\langle r^2,rs\rangle$, then we have a $G$-orbit consisting of three pairwise non-intersecting $(-1)$-curves. In the first two cases, one has $\Pic(S)^G\simeq\ZZ$. Finally, the claim about fixed points follows from Lemma~\ref{lem: dP6 fixed points}.
\end{proof}

\begin{remark}\label{rem: dP6 non-minimal S3}
	Note that, although $\Dih_6=\langle r,s\rangle$ contains two groups isomorphic to $\Sym_3$, only one of them gives $G$-invariant Picard number 1, namely $\langle r^2,s\rangle$, which we denote by $\Sym_3^{\rm min}$. The group $\langle r^2,rs\rangle$ will be denoted by $\Sym_3^{\rm nmin}$. Note that this is the quotient of $\Dih_6$ by its centre $\Center(\Dih_6)=\langle r^3\rangle \simeq\Cyc_2$.
\end{remark}

By \cite[Propositions 7.12 and~7.13]{DolgachevIskovskikh}, every Sarkisov link starting from $S$ is of type II and is represented by the diagram (\ref{eq: Sarkisov II}), where $\eta$ blows up a point of degree $d$ and one of the following holds:
\begin{enumerate}
	\item\label{b-1} $S\simeq S'$, $d=5$, $\chi$ is a birational Bertini involution;
	\item\label{b-2} $S\simeq S'$, $d=4$, $\chi$ is a birational Geiser involution;
	\item\label{b-3} $d=3$, $K_{S'}^2=6$;
	\item\label{b-4} $d=2$, $K_{S'}^2=6$;
	\item\label{b-5} $d=1$, $S'\simeq\PP^1\times\PP^1$.
\end{enumerate}

Let us emphasize that in cases~\ref{b-3} and~\ref{b-4}, the surfaces $S'$ \emph{does not} have to be $G$-isomorphic\footnote{Therefore, the user should be careful when using the statement of \cite[Proposition 7.13]{DolgachevIskovskikh}, whose notation is a bit misleading, in our opinion. For example, in the case of del Pezzo surfaces of degree 6 and links at points of degree 3 and 2, the authors write $S'\simeq S$, which might create the impression that this is an isomorphism of $G$-surfaces; compare with the case of del Pezzo surfaces of degree 8 and points of degree 4, where it is not written that $S'\simeq S$.} to $S$, as the following example shows. I am grateful to Andrey Trepalin for pointing this out.

\begin{example}\label{ex: dP6}
	Let $\omega$ be a primitive $\suprd{3}$ root of unity, and consider the finite subgroup $G\subset\Bir(\PP_\kk^2)$ generated by the following three elements:
	\[
	\alpha\colon [x:y:z]\longmapsto [x:\omega y:\omega^2 z],\quad \beta\colon [x:y:z]\longmapsto [y:z:x],\quad \gamma\colon [x:y:z]\longmapsto [yz:xz:xy].	
	\]
	One has 
	\[
	G=\left\langle\alpha,\beta,\gamma\ |\ \alpha^3=\beta^3=\gamma^2=\id,\ \alpha\beta=\beta\alpha,\ \beta\gamma=\gamma\beta,\ \gamma\alpha\gamma=\alpha^{-1} \right\rangle \simeq(\Cyc_3\times\Cyc_3)\rtimes\Cyc_2,
	\]
	where the copies of $\Cyc_3$ are generated by $\alpha$ and $\beta$, and $\Cyc_2$ is generated by the Cremona involution $\gamma$ and acts on $\Cyc_3\times\Cyc_3$ by coordinate exchange (\textit{i.e.}~we have the wreath product $G\simeq\Cyc_3\wr\Cyc_2$). The group $G$ is regularized on a del Pezzo surface $S$ of degree $6$ given by Equation (\ref{eq: del Pezzo of degree 6}), which we identify with the blow-up of $\PP^2$ in $p_1=[1:0:0]$, $p_2=[0:1:0]$ and $p_3=[0:0:1]$. The homomorphism $\psi\colon\Aut(S)\to\Dih_6$ induces a short exact sequence
	\begin{equation}\label{eq: toric 1}
	1\lra\langle\alpha\rangle\lra G\lra\langle\beta,\gamma\rangle\lra 1,
	\end{equation}
	where $\beta$ acts by permutation of coordinates in each triple $x_0,x_1,x_2$ and $y_0,y_1,y_2$, the involution~$\gamma$ acts as in (\ref{eq: del Pezzo 6 Cremona}) and $\alpha$ acts as in (\ref{eq: del Pezzo 6 torus action}). In particular, $\psi(G)\simeq\Cyc_6$. Consider the lift of the three points $p_4=[1:1:1]$, $p_5=[1:\omega:\omega^2]$ and $p_6=[1:\omega^2:\omega]$ on $S$. They form a $G$-orbit in general position on $S$. A Sarkisov link (\ref{eq: Sarkisov II}) centred at these points leads to a del Pezzo surface $S'$ of degree 6. Let us denote by $L_{ij}$ the strict transforms on the cubic surface $T$ of the lines on $\PP^2$ passing through $p_i$ and $p_j$. Similarly, $Q_{ijklr}$ will denote the strict transform of the conic passing through $p_i,p_j,p_k,p_l,p_r$. Then the morphism $\eta'$ blows down $Q_{12456}$, $Q_{23456}$ and $Q_{13456}$. The hexagon of $(-1)$-curves on $S'$ consists of the $\eta'$-images of $L_{45}$, $L_{46}$, $L_{56}$, $Q_{12346}$, $Q_{12345}$ and $Q_{12356}$. Hence the homomorphism $\psi'\colon\Aut(S')\to\Dih_6$ induces a short exact sequence
	\begin{equation}\label{eq: toric 2}
	1\lra\langle\beta\rangle\lra G\lra\langle\alpha,\gamma\rangle\lra 1.
	\end{equation}
	In particular, $\psi'(G)\simeq\Sym_3$. But this clearly implies that $S$ and $S'$ cannot be $G$-isomorphic.
\end{example}

So, one has to pay  special attention to Sarkisov links centred at points of degrees 3 and~2. Proposition~\ref{prop: dP6 links} below shows that the extra condition of being $H$-isomorphic eliminates the phenomena described in Example~\ref{ex: dP6}. To prove it, we will need some technical lemmas.

\begin{lemma}\label{lem: dP6 standard form of a group}
	Let $S$ be a del Pezzo surface of degree 6 and $\tau$ be the toric automorphism \eqref{eq: del Pezzo 6 torus action} with $\lambda_0=1$, $\lambda_1=\omega$, $\lambda_2=\omega^2$, where $\omega$ is a primitive $\suprd{3}$ root of unity. Assume that a group $\Gamma\subset\Aut(S)$ fits into the short exact sequence 
	\begin{equation}\label{eq: dP6 Gamma exact sequence}
	1\overset{}{\longrightarrow}\Gamma' \overset{}{\longrightarrow}\Gamma\overset{\psi}{\longrightarrow}\Gamma''\overset{}{\longrightarrow} 1,
	\end{equation}
	where $\Gamma'=\langle \tau\rangle$ or $\Gamma'=\id$. Then one has the following:
	\begin{enumerate}
		\item If\, $\Gamma''\simeq\Dih_6$, then $\Gamma$ is conjugate in $\Aut(S)$ to the subgroup $\Gamma_0\subset\Aut(S)$ generated by $\tau$, $\varrho$ and $\sigma$ if\, $\Gamma'=\langle \tau\rangle$, and by $\varrho$ and $\sigma$ if\, $\Gamma'=\id$. 
		\item If\, $\Gamma''\simeq\Cyc_6$, then $\Gamma$ is conjugate in $\Aut(S)$ to the subgroup $\Gamma_0\subset\Aut(S)$ generated by $\tau$ and $\varrho$ if\, $\Gamma'=\langle \tau\rangle$, and by $\varrho$ if\, $\Gamma'=\id$. 
	\end{enumerate}
\end{lemma}
\begin{proof}
	We first prove the claim for $\Gamma'=\langle \tau\rangle$ and  $\Gamma''\simeq\Dih_6$. Let $\overline{\varrho}\in\Gamma$ and $\overline{\sigma}\in\Gamma$ be such that $\psi(\overline{\varrho})$ and $\psi(\overline{\sigma})$ generate $\Gamma''$. We may thus assume that $\psi(\overline{\varrho})=\psi(\varrho)$ and $\psi(\overline{\sigma})=\psi(\sigma)$. Then $\Gamma$ is generated by $\tau$, $\overline{\varrho}$ and $\overline{\sigma}$. The map $\overline{\varrho}$ is given by
	\begin{equation}\label{eq: dP6 order 6}
		[(1,a,b)]\circ\varrho=\overline{\varrho}\colon ([x_0:x_1:x_2],[y_0:y_1:y_2])\longmapsto ([y_1:ay_2:by_0],[x_1:a^{-1}x_2:b^{-1}x_0] )
	\end{equation}
	for some $a,b\in\kk^*$. The map 
	\[
	\beta\colon([x_0:x_1:x_2],[y_0:y_1:y_2])\longmapsto ([x_0:ba^{-1}x_1:a^{-1}x_2],[y_0:ab^{-1}y_1,ay_2])
	\]
	commutes with $\tau$ and satisfies $\beta\circ\overline{\varrho}\circ\beta^{-1}=\varrho$. Let $\widetilde{\varrho}=\beta\circ\overline{\varrho}\circ\beta^{-1}=\varrho$ and $\widetilde{\sigma}=\beta\circ\overline{\sigma}\circ\beta^{-1}$. Then $\tau$, $\widetilde{\varrho}$, $\widetilde{\sigma}$ generate the group $\beta\circ\Gamma\circ\beta^{-1}$. Since $\beta\in\ker\psi=T$, we have $\psi(\widetilde{\sigma})=\psi(\sigma)$ and thus
	\begin{equation}\label{eq: dP6 order 2}
		\widetilde{\sigma}=\mu\circ\sigma \colon ([x_0:x_1:x_2],[y_0:y_1:y_2])\longmapsto ([y_0:cy_2:dy_1],[x_0:c^{-1}x_2:d^{-1}x_1])
	\end{equation}
	for some $\mu=[(1,c,d)]\in T$. Therefore, 
	\begin{equation}\label{eq: dP6 order 2 bis}
	\widetilde{\sigma}^2\colon ([x_0:x_1:x_2],[y_0:y_1:y_2])\longmapsto ([x_0:cd^{-1}x_1:c^{-1}dx_2],[y_0:c^{-1}dy_1:cd^{-1}y_2])
	\end{equation}
	is a power of $[(1,\omega,\omega^2)]$; hence $c=d\omega^i$ for  $i\in\{0,1,2\}$. Since $(\sigma\circ\varrho)^2=\id$, we have that 
	\begin{equation}\label{eq: dP6 order 2 bis bis}
	(\widetilde{\sigma}\circ\widetilde\varrho)^2\colon ([x_0:x_1:x_2],[y_0:y_1:y_2])\longmapsto ([cx_0:cx_1:d^2x_2],[c^{-1}y_0:c^{-1}y_1:d^{-2}y_2])
	\end{equation}
	is a power of $\tau=[(1,\omega,\omega^2)]$; hence $c=d^2$. Since $c=d\omega^i$, we conclude that $\mu=\tau^k$ for some $k\in\{0,1,2\}$, \textit{i.e.}~$\widetilde{\sigma}=\sigma\circ\tau^k$. Therefore, $\beta\circ\Gamma\circ\beta^{-1}=\langle\tau,\varrho,\sigma\circ\tau^k\rangle=\Gamma_0$, as  claimed. If $\Gamma''\simeq\Cyc_6$, then it is enough to conjugate the generator $\overline{\varrho}$ to $\varrho$.  
	
	If $\Gamma'=\id$ and $\Gamma''\simeq\Dih_6$ or $\Gamma''\simeq\Cyc_6$, we again conjugate the generator $\overline{\varrho}$ to $\varrho$. Since $\widetilde{\sigma}$ and $\widetilde{\sigma}\circ\widetilde{\varrho}$ are just involutions, (\ref{eq: dP6 order 2 bis}) gives $c=d$, while (\ref{eq: dP6 order 2 bis bis}) implies $c=d^2$. Hence $c=d=1$ and we are done. We refer to \cite[Propositions 5.6, 5.7 and~5.8]{Pinardin} for similar proofs in these cases.
\end{proof}	

Let us fix some notation. Let $\chi\colon S_1\dashrightarrow S_2$ be a Sarkisov $G$-link between del Pezzo surfaces of degree 6 and $H\subset G$ be a subgroup. We denote by $G_1=\iota_1(G)$ and $H_1=\iota_1(H)$ the embeddings of $G$ and $H$ into $\Aut(S_1)$, and by $G_2=\iota_2(G)$ and $H_2=\iota_2(H)$ the embeddings of $G$ and $H$ into $\Aut(S_2)$ induced by the map~$\chi$. For each $i\in\{1,2\}$, we denote by $\psi_i\colon\Aut(S_i)\to\Aut(\Sigma_i)\simeq\Dih_6$ the homomorphism described above with $T_i=\ker\psi_i$; further, set $G_{i,T}=G_i\cap T_i$, $H_{i,T}=H_i\cap T_i$, $\widehat{G}_i=\psi_i(G_i)$ and $\widehat{H}_i=\psi_i(H_i)$. We have $H_{i,T}\subset G_{i,T}$ and $\widehat{H}_i\subset\widehat{G}_i$ for each $i\in\{1,2\}$.

We start with some restrictions on the ``toric part'' of our groups and then deal with the simplest case when this part of the larger group $G$ is trivial.

\begin{lemma}\label{lem: dP6 trivial toric part}
	Let $S_1$ be a del Pezzo surface of degree 6 and $S_1\dashrightarrow S_2$ be a Sarkisov $G$-link of type II, centred at a $G$-point of degree $d$, where $d\in\{2,3\}$. Let $H\subset G$ be a subgroup, and assume there is an $H$-link $S_1\dashrightarrow S_2$ at the same point. Then the following hold:
	\begin{enumerate}
		\item\label{dP6 trivial toric part 1} The subgroups $H_{i,T}$ and $G_{i,T}$ are either trivial or of order $d$; \textit{i.e.}~for each $i$ we have the following possibilities:
		\begin{enumerate}
			\item $H_{i,T}=G_{i,T}=\id$, 
			\item $H_{i,T}=\id\subset G_{i,T}\simeq\Cyc_d$, 
			\item $H_{i,T}=G_{i,T}\simeq\Cyc_d$.
		\end{enumerate}
		\item\label{dP6 trivial toric part 2} Assume that $S_2$ is $H$-isomorphic to $S_1$. If\, $G_{1,T}=\id$, then $S_2$ is $G$-isomorphic to $S_1$.
		\item\label{dP6 trivial toric part 3} If\, $G_{1,T}\simeq\Cyc_d$,  then $G_{2,T}\simeq\Cyc_d$.
	\end{enumerate}
\end{lemma}
\begin{proof}
	Since each $S_i$ is $H$-del Pezzo, by Lemma~\ref{lem: dP6 minimal groups} the pair $(\widehat{H}_i,\widehat{G}_i)$ must be one of the following: $(\Cyc_6,\Cyc_6)$, $(\Sym_3,\Sym_3)$, $(\Dih_6,\Dih_6)$, $(\Cyc_6,\Dih_6)$, $(\Sym_3,\Dih_6)$. Furthermore, since $S_i$ admits an $H_i$-orbit of degree $d$ (which is also a $G_i$-orbit), both $H_i$ and $G_i$ have index $d$ subgroups which fix a point on $S_i\setminus\Sigma_i$ and thus do not intersect $H_{i,T}$ and $G_{i,T}$, respectively. We observe that the order of $G_{i,T}$ is at most $d$ and deduce statement~\ref{dP6 trivial toric part 1}.
	
	\ref{dP6 trivial toric part 2} Suppose that $G_{1,T}=\id$; then $H_1\simeq\widehat{H}_1$, $G_1\simeq\widehat{G}_1$, and the statement is tautological when $\widehat{H}_1=\widehat{G}_1$. If $\widehat{H}_1\ne\widehat{G}_1$, then $G_1\simeq\widehat{G}_1=\psi_1(\Aut(S_1))\simeq\Dih_6$. Since $G_1\simeq G_2$, we must have $G_2\simeq\widehat{G}_2$. Indeed, otherwise $G_{2,T}=G_2\cap T_2\ne\id$ and thus $G_{2,T}\simeq\Cyc_d$. But then $\widehat{G}_2\simeq\Cyc_2\times\Cyc_2$ (if $d=3$) or $\widehat{G}_2\simeq\Sym_3$ (if $d=2$). In the former case, $S_2$ is not $G$-del Pezzo by Lemma~\ref{lem: dP6 minimal groups}, while in the latter case $G_{2,T}\simeq\Cyc_2$ is necessarily the centre of $G_2\simeq\Cyc_2\times\Sym_3$; however, an involution from $T_2$ cannot commute with an automorphism $\tau'\circ\theta$, where $\tau'\in T_2$, mapped to a rotation of order 3. Now, having $G_1\simeq\psi_1(\Aut(S_1))$ and $G_2\simeq\psi_2(\Aut(S_2))$, both groups are conjugate to the ``standard'' one generated by $\varrho$ and $\sigma$; see Lemma~\ref{lem: dP6 standard form of a group}.
	
	\ref{dP6 trivial toric part 3} Assume that $G_{1,T}\simeq\Cyc_d$. If $G_{2,T}=\id$, then $G_1\simeq G_2\simeq\widehat{G}_2\in\{\Cyc_6,\Sym_3^{\rm min},\Dih_6\}$. If $d=3$, then $\widehat{G}_1\in\{\Cyc_2,\Cyc_2^2\}$ and thus $S_1$ is not $G$-del Pezzo. If $d=2$, then $\widehat{G}_1\simeq\Cyc_3$ or $G_{1,T}\simeq\Cyc_2$ is the centre of $G_1\simeq\Cyc_2\times\Sym_3$. In the first case, $S_1$ is not $G$-del Pezzo, while the second case is impossible, as was noticed before.
\end{proof}

\medskip

\begin{proposition}\label{prop: dP6 links}
	Let $S_1$ be a del Pezzo surface of degree 6 and $S_1\dashrightarrow S_2$ be a Sarkisov $G$-link of type II, centred at a $G$-point of degree $d$, where $d\in\{2,3\}$. Let $H\subset G$ be a subgroup, and assume there is an $H$-link $S_1\dashrightarrow S_2$ at the same point. If\, $S_2$ is $H$-isomorphic to $S_1$, then it is also $G$-isomorphic to $S_1$.
\end{proposition}

\begin{proof}

By Lemma~\ref{lem: dP6 trivial toric part},  we may assume $G_{1,T}\simeq\Cyc_d$. Let $\alpha\colon S_1\iso S_2$ be an $H$-isomorphism from the statement; it maps $T_1=S_1\setminus\Sigma_1$ isomorphically onto $T_2=S_2\setminus\Sigma_2$ and induces isomorphisms $H_{1,T}\simeq H_{2,T}$ and $\widehat{H}_1\simeq\widehat{H}_2$. Define $G_2'=\alpha\circ G_1\circ\alpha^{-1}\subset\Aut(S_2)$. We claim that $G_2'=\gamma^{-1}\circ G_2\circ\gamma$ for some $\gamma\in\Aut(S_2)$. This will finish the proof as the surfaces $S_1$ and $S_2$ are then $G$-isomorphic via the map $\gamma\circ\alpha$. 

Set $\widehat{G}_2'=\psi_2(G_2')$ and $G_{2,T}'=G_2'\cap T_2$. Firstly, note that  $\widehat{G}_2=\widehat{G}_2'$. Since $\widehat{G}_2'\simeq\widehat{G}_1$ and there is only one subgroup of $\Dih_6$ in each isomorphism class for which $S_2$ is $G$-del Pezzo, it is enough to show that $\widehat{G}_2\simeq\widehat{G}_1$. Both are subgroups of $\Dih_6$ of the same order, as $G_1\simeq G_2$ and $G_{1,T}\simeq G_{2,T}$ by Lemma~\ref{lem: dP6 trivial toric part}\ref{dP6 trivial toric part 3}; hence it remains to see that one cannot have $\widehat{G}_1\simeq\Cyc_6$ and $\widehat{G}_2\simeq\Sym_3$ (or vice versa); but this is implied by $\widehat{H}_1\simeq\widehat{H}_2$. 

Secondly, we claim that $G_{2,T}=G_{2,T}'$. Recall that the group $\psi_2(\Aut(S_2))\simeq\Sym_3\times\Cyc_2$ acts on $T_2$; namely, $\Sym_3$ acts on the torus $T_2\simeq(\kk^*)^3/\kk^*$ by permuting the coordinates, and the action of $\Cyc_2$ is the inversion; we denote this action of $\Dih_6$ on $T_2$ by $\varphi$. The groups $\widehat{G}_2'$ and $\widehat{G}_2$, being both $\Sym_3^{\rm min}$, $\Cyc_6$ or $\Dih_6$, contain $g=r^2$. Let $\tau=(t,\id)\in T_2\rtimes\rho_2(\Aut(S_2))$ be an element of order $d\in\{2,3\}$ which generates $G_{2,T}$ (or $G_{2,T}'$) and $(u,g)$ be an element of $G_2$ (respectively, of $G_2'$) which is mapped to $(\id,g)$ by $\psi_2$. Then $(u,g)^{-1}=(\varphi_{g^{-1}}(u^{-1}),g^{-1})$. Therefore, $(\varphi_{g^{-1}}(u^{-1}),g^{-1})(t,\id)(u,g)=(\varphi_{g^{-1}}(u^{-1}tu),1)=(\varphi_{g^{-1}}(t),1)$ is a power of $\tau=(t,\id)$. If $t=[(1,a,b)]\in T_2\simeq(\kk^*)^3/\kk^*$, we must have $[(a,b,1)]=[(1,a,b)]$ or $[(a,b,1)]=[(1,a^2,b^2)]$, where $a$ and $b$ are primitive $\supth{d}$ roots of unity. This implies that $d=3$ and $\tau=([(1,\omega,\omega^2)],\id)$ or $\tau=([(1,\omega^2,\omega)],\id)$, which both generate the same subgroup of $T_2$. 

We conclude by applying Lemma~\ref{lem: dP6 standard form of a group}. Namely, the groups $G_2$ and $G_2'$ are both conjugate to $\langle\tau,\varrho,\sigma\rangle$ or to $\langle\tau,\varrho\rangle$ if their $\psi_2$-images are $\Dih_6$ or $\Cyc_6$, respectively, or to the group $\langle\tau,H_2\rangle$ if their $\psi_2$-images are $\Sym_3^{\min}$.
\end{proof}

\begin{proposition}
	If a del Pezzo surface of degree 6 is $H$-birationally rigid, then it is also $G$-birationally rigid.
\end{proposition}
\begin{proof}
Possible Sarkisov $G$-links were described before Example~\ref{ex: dP6}; recall that $G$-birational Bertini or Geiser involutions always lead to a $G$-isomorphic surface. If $S$ admits a $G$-link to $\PP^1\times\PP^1$ centred at a $G$-fixed point, then the same point is fixed by $H$ and hence there is an $H$-link to $\PP^1\times\PP^1$, so we have a contradiction.  Assume that there is a $G$-link $S\dashrightarrow S'$ at a $G$-orbit of cardinality 2 or 3. Then either this orbit  contains an $H$-fixed point, or it gives rise to an $H$-link $S\dashrightarrow S'$. In the former case, we get an $H$-link to $\PP^1\times\PP^1$, contradicting the $H$-birational rigidity of $S$, while in the second case, the $H$-birational rigidity of $S$ implies that $S'$ is $H$-isomorphic to $S$. Now Proposition~\ref{prop: dP6 links} shows that $S'$ is also $G$-isomorphic to $S$.
\end{proof}

\begin{remark}\label{rem: Iskovskikh example}
	If there is a $G$-link $S\dashrightarrow\PP^1\times\PP^1$ centred at a $G$-fixed point, then $G\cap T=\id$ by Lemma~\ref{lem: dP6 fixed points} and hence $\psi$ maps $G$ isomorphically onto one of the following subgroups of $\psi(\Aut(S))$: $\Cyc_6$, $\Sym_3^{\rm min}$ or $\Dih_6$.  If their actions are in the ``standard'' form as in Lemma~\ref{lem: dP6 standard form of a group}, the fixed point becomes $([1:1:1],[1:1:1])$. Making a Sarkisov link centred at this point, we arrive at $\PP^1\times\PP^1$ acted on by $G$. Explicitly, one blows up the fixed point and then blows down the preimages of the three genus zero curves passing through this point; see \cite{IskovskikhNonConjugate} for more details and related discussion.
\end{remark}

\section{\texorpdfstring{$\boldsymbol{G}$}{G}-birational rigidity of quadric surfaces}\label{sec: dP8}

In this section, we investigate carefully the case of $S=\PP^1\times\PP^1$. In Section~\ref{subsec: dP8 G-links}, we study some Sarkisov links on $S$ and show how to complete the proof of our main theorem. In Section~\ref{subsec: dP8 complete classification}, we present a detailed analysis of finite group actions on $S$ and $G$-birational rigidity in each case, more in the spirit of \cite{SakovicRigidity,WolterEquivariantBirationalDP5}.

\subsection{Sarkisov $\boldsymbol{G}$-links}\label{subsec: dP8 G-links} By \cite[Propositions 7.12 and~7.13]{DolgachevIskovskikh}, every Sarkisov $G$-link starting from $S$ is either of type I and of the form 
\begin{equation}\label{eq: type I link}
\xymatrix{
	&& T\ar@{->}[dll]_{\eta}\ar@{->}[d]^{\pi}\\
	S\ar[d]&& \PP^1\ar[dll]\\
	\pt\rlap{,} &&
}
\end{equation}
where $\eta$ is a blow-up of a degree 2 point, $\pi\colon T\to \PP^1$ is a $G$-conic bundle with two singular fibres (in fact, it is necessarily a del Pezzo surface of degree 6, see \cite[Theorem 5]{Iskovskikh_minimal}) or is of type II and is represented by the diagram (\ref{eq: Sarkisov II}), where $\eta$ blows up a point of degree $d$, $\eta'$ blows up a point of degree $d'$, and one of the following holds:
\begin{enumerate}
	\item $S\simeq S'$, $d=d'=7$, $\chi$ is a birational Bertini involution;
	\item $S\simeq S'$, $d=d'=6$, $\chi$ is a birational Geiser involution;
	\item $S'$ is a del Pezzo surface of degree 5, $d=5$, $d'=2$;
	\item $S'\simeq\PP^1\times\PP^1$, $d=d'=4$;
	\item $S'$ is a del Pezzo surface of degree 6, $d=3$, $d'=1$;
	\item $S'\simeq\PP^2$, $d=1$, $d'=2$.
\end{enumerate}
In particular, if $G$ does not fix a point on $S$ and is not isomorphic to any of the groups
\begin{equation}\label{eq: dP8 groups}
	\Cyc_5,\quad \Cyc_6,\quad \Sym_3,\quad\Dih_5,\quad\Dih_6,\quad\GA_1(\FF_5),\quad\Alt_5,\quad\Sym_5,
\end{equation}
then all $G$-links of type II from $S$ lead to a quadric surface $S'\simeq\PP^1\times\PP^1$ (recall Lemmas~\ref{lemma: dP5 minimal groups},~\ref{lem: dP6 fixed points} and~\ref{lem: dP6 minimal groups}). \emph{A priori}, the surface $S'$ does not have to be $G$-isomorphic to $S$ (we saw such phenomena in Section~\ref{sec: dP6}) unless we deal with $G$-birational Bertini and Geiser involutions. It turns out that $G$-links centred at a point of degree 4 also lead to a $G$-isomorphic del Pezzo surface $S'\simeq\PP^1\times\PP^1$. As was pointed out to me by Andrey Trepalin, this holds in the arithmetic, see \cite[Lemma 4.3]{TrepalinDP8},  and even mixed settings; we limit ourselves to the geometric situation. 

\begin{proposition}\label{prop: dP8 link at degree 4}
	Let $\chi\colon S\dashrightarrow S'$ be a Sarkisov $G$-link of type II centred at a point of degree~4. Then $S'$ is $G$-isomorphic to $S$.
\end{proposition}
\begin{proof}
	Assume that $\chi=\eta'\circ\eta^{-1}\colon S\dashrightarrow S'$ is given by the diagram (\ref{eq: Sarkisov II}); then $T$ is a del Pezzo surface of degree 4. It is a blow-up of $\PP_\kk^2$ in five points $p_1,\ldots,p_5$ in general position. Denote by $E_1,\ldots,E_5$ the exceptional divisors of this blow-up, by $L_{ij}$ for $i,j\in\{1,\ldots,5\}$ with $i<j$ the strict transforms of the lines through $p_i$ and $p_j$, and by $Q$ the strict transform of the conic through $p_1,\ldots,p_5$. These are sixteen $(-1)$-curves on $T$. Their intersection graph is the Clebsch strongly regular quintic graph on sixteen vertices, shown on~Figure~\ref{figure: Clebsch}. Up to renumbering the points, we may assume that the curves $\Sigma=\{E_1,E_2,E_3, L_{45}\}$ are the exceptional divisors of the blow-up $\eta$, while the curves $\Sigma'=\{L_{12},L_{13},L_{23},Q\}$ are the exceptional divisors of $\eta'$. Recall, see \cite[Corollary~8.2.40]{DolgachevClassicalAG}, that $\Aut(T)$ injects into the Weyl group $\WW(D_5)\simeq\Cyc_2^4\rtimes\Sym_5$, the automorphism group of the Clebsch graph. The quartic del Pezzo surface $T$ is isomorphic to an intersection of two quadrics
	\[
	\sum_{i=1}^{5}x_i^2=\sum_{i=1}^{5}\lambda_ix_i^2=0
	\]
	in $\PP_\kk^4$. The group $\Sym_5$ acts by permutation of coordinates (and naturally acts on the indices of $E_i$ and $L_{jk}$), while $\Cyc_2^4$ acts as a diagonal subgroup of $\PGL_5(\kk)$. There are two types of involutions in this group, $\iota_{ij}$ and $\iota_{ijkl}$, which switch the signs of $x_i,x_j$ and $x_i,x_j,x_k,x_l$, respectively. Equivalently, $\iota_{ijkl}$ coincides with the automorphism $\jmath_t$, where $t\in\{1,2,3,4,5\}\setminus\{i,j,k,l\}$, switching the sign of $x_t$. As explained in \cite[Section~7]{DolgachevDuncanFixed}, the $\jmath_t$ are given by de Jonqui\`{e}res involutions of the plane model centred at $p_t$ and interchange $E_i$ with $L_{it}$ and $E_t$ with $Q$. To recover the action of $\jmath_t$ on the curves $L_{ij}$ for $i,j\ne t$, we notice that two disjoint curves $E_i$ and $E_j$ are both intersected by exactly two others, $L_{ij}$ and $Q$ (in other words, any two non-adjacent vertices of the Clebsch graph~\ref{figure: Clebsch} have two common neighbours). Therefore, $\jmath_t(L_{ij})$ is the curve which intersects $L_{it}$ and $L_{jt}$, and it is different from $\jmath_t(Q)=E_t$; hence, it is $L_{sr}$, where $s,r\in\{ 1,2,3,4,5\}\setminus\{i,j,t\}$.
	
	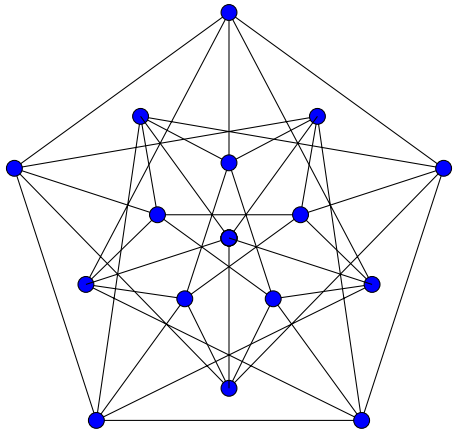
\begin{figure}[H]
		\begin{center}
			\begin{tikzpicture}
				\draw (18:1cm) -- (162:1cm) -- (306:1cm) -- (90:1cm) -- (234:1cm) -- cycle;
				\draw (18:3cm) -- (90:3cm) -- (162:3cm) -- (234:3cm) -- (306:3cm) -- cycle;
				\draw (18:1cm) -- (54:2cm) --(90:1cm) -- (126:2cm) -- (162:1cm) -- (198:2cm) -- (234:1cm) -- (270:2cm) -- (306:1cm) -- (342:2cm) -- cycle;
				\draw (18:3cm) -- (126:2cm) -- (234:3cm)-- (342:2cm) -- (90:3cm) -- (198:2cm) -- (306:3cm) -- (54:2cm) -- (162:3cm) -- (270:2cm) -- cycle;
				\foreach \x in {18,90,162,234,306}
				{\draw (\x:1cm) -- (\x:3cm);
					\draw[black,fill=blue] (\x:3cm) circle (3pt);
					\draw[black,fill=blue] (\x:1cm) circle (3pt);}
				\foreach \x in {54, 126, 198, 270, 342}
				{\draw[black,fill=blue] (\x:2cm) circle (3pt);
					\draw[black,fill=blue] (\x:0cm) circle (3pt);
					\draw (\x:0cm) -- (\x:2cm);}
			\end{tikzpicture}
		\end{center}
		\caption{The Clebsch graph}
		\label{figure: Clebsch}
	\end{figure}
	
	Now, one easily checks that the involutions $\iota_{*5}=\jmath_*\circ\jmath_5$ and $\iota_{*4}=\jmath_*\circ\jmath_4$ do not preserve the set $\Sigma$. Hence, the subgroup of $\WW(D_5)$ which preserves $\Sigma$ is generated by $(12),(123),\iota_{12}$ and $(45)$, and is isomorphic to $\Sym_4\times\Cyc_2$. The involution $\iota_{45}$ commutes with this group and maps $\Sigma$ onto $\Sigma'$. Since $\iota_{45}$ actually corresponds to an automorphism of $T$, we conclude that the blow-down $\eta'$ yields a $G$-isomorphic del Pezzo surface $S'$. 
\end{proof}

\begin{corollary}\label{cor: dP8 rigidity and orbits}
	The surface $\PP^1\times\PP^1$ is $G$-birationally rigid $($as a $G$-del Pezzo surface$)$ if and only if the size of every $G$-orbit in general position is $4$ or at least~$6$. $($Here and everywhere below, by ``general position'' we mean that the blow-up of this orbit gives a del Pezzo surface.$)$
\end{corollary}
\begin{proof}
	The sufficiency follows from Proposition~\ref{prop: dP8 link at degree 4} and the fact that $G$-birational involutions of Bertini and Geiser yield a $G$-isomorphic surface. Conversely, if $\PP^1\times\PP^1$ is $G$-birationally rigid, then it does not admit Sarkisov $G$-links centred at $G$-points of degree $d\in\{1,2,3,5\}$: by our assumption, the blow-up of every such orbit gives a del Pezzo surface $T$ with $\Pic(T)^G\simeq\ZZ^2$; hence the 2-ray game (see Remark~\ref{rem: rank r fibration}) provides a Sarkisov $G$-link to a $G$-del Pezzo surface of degree $d'\ne 8$.
\end{proof}

\begin{remark}\label{rem: dP4 auto}
	At the same time, it is often possible to exclude the possibility of $G$-links $\chi\colon S\dashrightarrow\PP^1\times\PP^1$ centred at a point of degree 4. If such a link exists and is represented by the diagram (\ref{eq: Sarkisov II}), then $T$ is a del Pezzo surface of degree 4, so a natural obstruction to the existence of $\chi$ is the impossibility of an embedding $G\hookrightarrow\Aut(T)$. Luckily, all possible automorphism groups $\Aut(T)$ of smooth del Pezzo surfaces of degree 4 are classified, see \cite{HosohQuarticDP,DolgachevClassicalAG}, and they are the following (see the proof of Proposition~\ref{prop: dP8 link at degree 4} for the description of these semidirect products):
	\[
	\Cyc_2^4,\quad \Cyc_2^4\rtimes\Cyc_2,\quad \Cyc_2^4\rtimes\Cyc_4,\quad\Cyc_2^4\rtimes\Sym_3,\quad \Cyc_2^4\rtimes\Dih_5.              
	\]  
\end{remark}

	The following proposition finishes the proof of our main theorem. An alternative way will be sketched in Remark~\ref{rem: proof bis}.
\begin{proposition}
	If the del Pezzo surface $\PP^1\times\PP^1$ is $H$-birationally rigid, then it is also $G$-birationally rigid.
\end{proposition}
\begin{proof}
	Suppose that $S=\PP^1\times\PP^1$ is not $G$-birationally rigid. By Corollary~\ref{cor: dP8 rigidity and orbits}, there is a $G$-orbit $\Sigma$ of size $|\Sigma|\in\{1,2,3,5 \}$; moreover, $\Sigma$ is in general position on $S$. Write $\Sigma=\Sigma_1\sqcup\cdots\sqcup\Sigma_r$, where each $\Sigma_i$ is an $H$-orbit and $|\Sigma_i|\leqslant |\Sigma_j|$ for $i\leqslant j$. Clearly, $r\geqslant 2$ since $r=1$ implies that $\Sigma$ is an $H$-orbit and hence $S$ is not $H$-birationally rigid. If $r\geqslant 2$ and $|\Sigma|\in\{1,2,3\}$, then $H$ admits an orbit of size 1, and thus $S$ has an $H$-link to $\PP^2$. The same reasoning applies to the case $|\Sigma|=5$ and $r\geqslant 3$, and we conclude that $r=2$, $|\Sigma_1|=2$, $|\Sigma_2|=3$. Since $\Sigma_1$ is in general position, Corollary~\ref{cor: dP8 rigidity and orbits} again gives a contradiction with $H$-birational rigidity. 
\end{proof}

The main theorem is proven.

\subsection{Finite groups acting on quadric surfaces}\label{subsec: dP8 complete classification} 
We now proceed with a deeper analysis of finite group actions on $S=\PP^1\times\PP^1$. Recall that
\[
\Aut(S)\simeq (\PGL_2(\kk)\times\PGL_2(\kk))\rtimes\Cyc_2,
\]
where the action of $\Cyc_2$ is given by exchanging the factors. Finite subgroups of the direct product can be determined using the so-called \emph{Goursat's lemma}. Recall that the \emph{fibred product} of two groups $G_1$ and $G_2$ over a group $D$ is 
\[
G_1\times_D G_2=\{(g_1,g_2)\in G_1\times G_2\colon \alpha(g_1)=\beta(g_2)\},
\] 
where $\alpha\colon G_1\to D$ and $\beta\colon G_2\to D$ are some surjective homomorphisms. Although the notation does not reflect it, the data defining $G_1\times_D G_2$ is not only
the groups $G_1$, $G_2$ and $D$ but also the homomorphisms $\alpha$,~$\beta$.

\begin{lemma}[{{Goursat's lemma}, \textit{cf.} \cite[p.~47]{Goursat}}]\label{lem: Goursat}
	Let $A$ and $B$ be two groups. There is a bijective correspondence between subgroups $G\subset A\times B$ and 5-tuples $\{G_A,G_B,K_A,K_B,\varphi\}$, where $G_A$ is a subgroup of $A$, $K_A$ is a normal subgroup of $G_A$, $G_B$ is a subgroup of $B$, $K_B$ is a normal subgroup of $G_B$ and $\varphi\colon G_A/K_A\iso G_B/K_B$ is an isomorphism. More precisely, the group corresponding to this 5-tuple is
	\[
	G=\{(a,b)\in G_A\times G_B\colon \varphi(aK_A)=bK_B\}.
	\]
	Conversely, let $G$ be a subgroup of $A\times B$. Denote by $p_A\colon A\times B\to A$ and $p_B\colon A\times B\to B$ the natural projections, and set $G_A=p_A(G)$ and $G_B=p_B(G)$. Further, let
	\begin{gather*}
	K_A=\ker p_B|_G=\{(a,\id)\in G,\ a\in A\},\\ K_B=\ker p_A|_G=\{(\id,b)\in G,\ b\in B\},
	\end{gather*}
	whose images by $p_A$ and $p_B$ define normal subgroups of $G_A$ and $G_B$, respectively $($denoted the same$)$. Let $\pi_A\colon G_A\to G_A/K_A$ and $\pi_B\colon G_B\to G_B/K_B$ be the canonical projections. The map $\varphi\colon G_A/K_A\to G_B/K_B$,  $\varphi(aK_A)=bK_B$, where $b\in B$ is any element such that $(a,b)\in G$, is an isomorphism. Furthermore, $G=G_A\times_D G_B$, where $D=G_A/K_A$, $\alpha=\pi_A$ and $\beta=\varphi^{-1}\circ\pi_B$.

\end{lemma}

\begin{corollary}\label{cor: Goursat}
	In the notation from Goursat's lemma, the subgroup $G\subset A\times B$ fits into the short exact sequence
	\[
	1\lra K_A\times K_B\lra G\lra D\lra 1.
	\]
\end{corollary}
\begin{proof}
	Indeed, the restriction of the homomorphism $\alpha\times\beta\colon G_A\times G_B\to D\times D$ to $G$ has  kernel $K_A\times K_B$, and its image is isomorphic to $\Delta=\{(t,t)\in D\times D\}\simeq D$.
\end{proof}

Using Goursat's lemma, one can get the description of finite subgroups $G\subset\Aut(\PP^1\times\PP^1)$ for which $\Pic(\PP^1\times\PP^1)^G\simeq\ZZ$. Before doing that, let us recall the following classical result due to F.~Klein. 

\begin{proposition}[\textit{cf.} {\cite{KleinIcosahedron}}]\label{prop: Klein}
	If\, $\kk$ is an algebraically closed field of characteristic zero, then every finite subgroup of\, $\PGL_2(\kk)$ is isomorphic to $\Cyc_n$, $\Dih_n$ $($where $n\geqslant 1)$, $\Alt_4$, $\Sym_4$ or $\Alt_5$. Moreover, there is only one conjugacy class for each of these groups. 
\end{proposition}

Every group $G\subset\Aut(\PP^1\times\PP^1)$ fits into the short exact sequence
\[
1\lra G_\circ\lra G\lra\widehat{G}\lra 1,
\]
where $G_\circ=G\cap (\PGL_2(\kk)\times\PGL_2(\kk))$ and $\widehat{G}\subseteq\Cyc_2$. 

\begin{proposition}[\textit{cf.} {\cite[Lemma 3.2]{TrepalinDPhigh}}]\label{prop: dP8 minimal groups}
	Let $G\subset\Aut(S)$ be a finite subgroup such that $\Pic(S)^G\simeq\ZZ$. Then $G\simeq (F\times_D F)_\bullet\Cyc_2$, where $F$ is cyclic, dihedral or one of the groups $\Alt_4$, $\Sym_4$, $\Alt_5$. Moreover, for every such group $G$,  we have $\Pic(S)^G\simeq\ZZ$.
\end{proposition}

\begin{proof}
	Since $G_\circ$ preserves the factors of $\PP^1\times\PP^1$, the condition $\Pic(S)^G\simeq\ZZ$ forces $\widehat{G}=\Cyc_2$. If $G_1$ and $G_2$ are the images of $G_\circ$ under the projections of $\PGL_2(\kk)\times\PGL_2(\kk)$ onto its factors, Goursat's lemma implies that $G_\circ= G_1\times_D G_2$ for some $D$. Since $\widehat{G}\ne\id$, we must have $G_1\simeq G_2$. Combining this with Proposition~\ref{prop: Klein}, we get the result.
\end{proof}

\begin{corollary}\label{cor: minimal groups structure}
	In the notation of Proposition~\ref{prop: dP8 minimal groups}, for every finite subgroup $G\subset\Aut(S)$ satisfying $\Pic(S)^G\simeq\ZZ$, one has the  short exact sequence
	\begin{equation}\label{eq: exact sequence Goursat}
	1\lra K\times K\lra G_\circ\lra D\lra 1, 
	\end{equation}
	where $K$ is a normal subgroup of\, $F$ and $F/K\simeq D$.
\end{corollary}
\begin{proof}
	We apply Corollary~\ref{cor: Goursat} and note that the action of $\Cyc_2$ on the semidirect product $(\PGL_2(\kk)\times\PGL_2(\kk))\rtimes\Cyc_2$ induces an isomorphism of the kernels $K_A\iso K_B$. (In fact, for finite subgroups $A,K_1,K_2$ of $\PGL_2(\kk)$, an isomorphism $A/K_1\simeq A/K_2$ always implies $K_1\simeq K_2$ unless $A=\Dih_{2n}$; see Lemma~\ref{lem: dihedral group subgroups}.)
\end{proof}

\begin{lemma}\label{lemma: structure of Gcirc}
	Let $G\subset\Aut(\PP^1\times\PP^1)$ be a finite subgroup. Then, the following holds:
	\begin{enumerate}
		\item\label{structure 1} Assume there is a Sarkisov $G$-link $\chi\colon S\dashrightarrow T$ of type I as in \eqref{eq: type I link}. Then $G_\circ$ has a subgroup of the form $\Cyc_n\times\Cyc_m$ which is of index at most $2$ in $G_\circ$. 
		\item\label{structure 2} Assume there exists a Sarkisov $G$-link $S\dashrightarrow\PP^2$ of type II. Then $G_\circ$ is isomorphic to a direct product of at most two cyclic groups. In particular, the abelian group $G_\circ$ is generated by at most two elements.
	\end{enumerate}
\end{lemma}
\begin{proof}
	\ref{structure 1} Recall that the centre of $\chi$ is a $G$-point $\{p,q\}$ of degree 2. Therefore, the stabilizer $G'\subset G$ of $p$ is a normal subgroup of index 2 and $G''=G_\circ\cap G'$ is of index at most 2 in $G_\circ$. Since $G''$ acts on $S$ fibrewise, the faithful representation $G''\to\GL(T_pS)$ is reducible; \textit{i.e.}~it is a direct sum of two 1-dimensional representations. Hence $G''\simeq\Cyc_n\times\Cyc_m$. In  case~\ref{structure 2}, the group $G$ fixes a point on $S$, and the result follows similarly.
\end{proof}

We now proceed with determining for which groups $G\simeq (F\times_D F)_\bullet\Cyc_2$ of Proposition~\ref{prop: dP8 minimal groups} the surface $S=\PP^1\times\PP^1$ is actually $G$-birationally rigid. In our analysis, we will often use Corollary~\ref{cor: minimal groups structure} (including its notation) without explicitly mentioning it.

\subsection{Case $\boldsymbol{F=\Cyc_n}$}\label{subsec: dP8 Cyc} Set $S=L_1\times L_2$, where $L_1\simeq L_2\simeq\PP^1$. Then $G_\circ\subseteq H_1\times H_2$ and $H_1\simeq H_2\simeq\Cyc_n$ are cyclic groups of order $n$, and $H_i$ faithfully acts on $L_i$. Recall that each $H_i$ fixes exactly two points on $L_i$. Hence, the group $G_\circ$ fixes four points on $S$. Let $p\in S$ be one of these fixed points. The group $G$ is a disjoint union of $G_\circ$ and $g G_\circ$, where $g\in G\setminus G_\circ$. Consider the set $\Omega=\{p,g(p)\}$. Since $g^2\in G_\circ$, this set is $g$-invariant. Furthermore, since $g^{-1}G_\circ g=G_\circ$, the point $g(p)$ is fixed by $G_\circ$. Therefore, $\Omega$ is $G$-invariant. If $|\Omega|=1$ (\textit{i.e.}~$g(p)=p$), then the stereographic projection from $p$ conjugates $G$ to a group acting on $\PP^2$, so $S$ is not $G$-birationally rigid. Assume that $g(p)\ne p$. Then $g(p)$ and $p$ are in general position on $S$. Indeed, the coordinates $[x:y]$ on $L_1$ and $[z:t]$ on $L_2$ can be chosen so that the fixed points of $G_\circ$ are
\[
([1:0],[1:0]),\ ([1:0],[0:1]),\ ([0:1],[1:0]),\ ([0:1],[0:1]).
\]
The automorphism $g$ is of the form $([x:y],[z:t])\mapsto (A[z:t],B[x:y])$, where $A,B\in\PGL_2(\kk)$. Assume that $g$ sends $([1:0],[1:0])$ onto $([1:0],[0:1])$, \textit{i.e.}~$A[1:0]=[1:0]$ and $B[1:0]=[0:1]$. Then, as discussed above, $g^2\colon ([x:y],[z:t])\mapsto (AB[x:y],BA [z:t] )$ fixes $([1:0],[1:0])$, which is impossible as $BA[1:0]=B[1:0]=[0:1]$. Similarly, $g$ cannot send $([1:0],[1:0])$ onto $([0:1],[1:0])$. We conclude that $p$ and $g(p)$ are not in a common fibre of either projection. Their blow-up gives a $G$-conic bundle $S'\to\PP^1$ with two singular fibres; hence $S$ is not even $G$-solid. 

\begin{remark}\label{rem: cyclic group quadric}
	As we already noticed in Remark~\ref{rem: fixed point cyclic}, a cyclic group always has a fixed point on a quadric. Blowing this point up and contracting the strict transforms of the lines passing through it, we arrive at $\PP^2$ (one often says that the group is \emph{linearizable} in this case).
\end{remark}

\subsection{Case $\boldsymbol{F=\Dih_n}$}\label{subsec: dP8 Dih} Recall that $\Dih_2\simeq\Klein_4\simeq\Cyc_2^2$, and set $\Dih_1=\Cyc_2$. Assume $n\geqslant 3$. Recall from Lemma~\ref{lem: dihedral group subgroups} that proper normal subgroups of $\Dih_n$ are cyclic groups of order $n/d$ for each $d$ dividing $n$ (of index $2d$) and, if $n$ is even, dihedral of index 2. Therefore, we have the following two possibilities:

\subsubsection{}\label{sec: series 2} The group $G_\circ$ fits into the short exact sequence
\begin{equation}\label{eq: dihedral 2}
	1\lra \Dih_m\times\Dih_m\lra G_\circ\lra D\lra 1,
\end{equation}
where either $D\simeq\Cyc_2$ and $n=2m$, $m\geqslant 2$, or $D=\id$ and $m=n\geqslant 3$.

\subsubsection{}\label{sec: series 1} The group $G_\circ$ fits into the short exact sequence
\begin{equation}\label{eq: dihedral 1}
	1\lra \Cyc_m\times\Cyc_m\lra G_\circ\lra\Dih_d\lra 1,
\end{equation}
where $n=md$, $m\geqslant 1$. Note that this includes the extremal case $m=1$, $G_\circ=\Dih_n$, $n\geqslant 3$.

\begin{proposition}\label{prop: dihedral groups on quadric}
	In the notation from above, one has the following:
	\begin{enumerate}
		\item\label{Dih 1} $S$ is $G$-birationally rigid in the case~\ref{sec: series 2}.
		\item\label{Dih 2} $S$ may fail to be $G$-birationally rigid in the case~\ref{sec: series 1}. If $\chi\colon S\dashrightarrow S'$ is a Sarkisov $G$-link to a different $G$-Mori fibre space, then either $S'=\PP^2$, or $S'$ is a $G$-del Pezzo surface of degree~6 and $G\simeq\Dih_6$, or $S'$ is a $G$-del Pezzo surface of degree 5 and $G\simeq\GA_1(\FF_5)$, or $S'$ admits the structure of a $G$-conic bundle with two singular fibres. 
	\end{enumerate}
\end{proposition}
\begin{proof}
	\ref{Dih 1} First, there are no Sarkisov $G$-links of type~I from $S$ or links to $\PP^2$. Indeed, otherwise, $G_\circ$ has a subgroup $N_\circ\simeq\Cyc_k\times\Cyc_l$ of index at most two by Lemma~\ref{lemma: structure of Gcirc}. But then $N_\circ\cap(\Dih_m\times\Dih_m)$ must be an abelian group which can be generated by at most two elements and has index at most~2 in $\Dih_m\times\Dih_m$, which is not possible. Furthermore, $G_\circ$ obviously does not embed into the groups from (\ref{eq: dP8 groups}). We are also able to show that there are no Sarkisov $G$-links of type II centred at a point of degree 4. Assume that it exists and is given by the diagram (\ref{eq: Sarkisov II}). Then $T$ is a del Pezzo surface of degree 4 in $\PP^4$. Note that $|G|=4n^2$ if $D\simeq\Cyc_2$, and $|G|=8n^2$ if $D=\id$. From Remark~\ref{rem: dP4 auto} we see that $G$ does not embed into $\Aut(T)$ unless $m=2$; in particular, $G$ must contain a subgroup of the form $\Delta=\Cyc_2^4$. It is easy to prove, see \cite[Lemma 3.1 and Proposition 3.11]{BeauvilleElementary}, that such a $\Delta$ is conjugate in $\PGL_5(\kk)$ to a subgroup of the diagonal torus; it acts on $T$ (which is an intersection of two quadrics in $\PP^4$) by changing the signs of the ambient coordinates $x_k$ and consists of the projective transformations $\id$, $\jmath_i$ for $i=1,\ldots,5$, $\jmath_i\circ\jmath_j$ for $1\leqslant i<j\leqslant 5$, where $\jmath_k\colon x_k\mapsto -x_k$. The fixed-point locus of each $\jmath_k$ is an elliptic curve cut out on $T$ by the hyperplane $\{x_k = 0\}$, while the fixed-point loci of other non-trivial involutions in $\Delta$ consist of exactly four points. If $\tr(\delta)$ denotes the trace of the action of $\delta$ on $\Pic(T)\otimes\CC$ and $\Eu(\cdot)$ denotes the topological Euler characteristic, then $\Eu(T^\delta)=\tr(\delta)+2$ by the Lefschetz fixed-point formula (here $T^\delta$ denotes the fixed locus of $\delta$). Since $\Eu(T^{\jmath_k})=0$ and $\Eu(T^{\jmath_k\circ\jmath_l})=4$, we get $\tr(\jmath_k)=-2$ and $\tr(\jmath_k\circ\jmath_l)=2$. Therefore, $\rank\Pic(T)^\Delta=\frac{1}{|\Delta|}\sum_{\delta\in\Delta}\tr(\delta)=1$, which contradicts the existence of a Sarkisov link; see \cite[Section~6.1]{DolgachevIskovskikh} or \cite[Section~2.1]{YasinskyDelPezzo} for more details. We conclude that $S$ is $G$-birationally rigid.
	
	\ref{Dih 2} All possible Sarkisov $G$-links from $S$ were described at the beginning of the section; as we know, those centred at points of degrees 7, 6 and 4 lead to  $G$-isomorphic del Pezzo surfaces. Suppose that $S'$ is a del Pezzo surface of degree 6. As $\chi^{-1}$ starts with blowing up a $G$-fixed point on $S'$, the group $G$ must be isomorphic to $\Cyc_6$, $\Sym_3$ or $\Dih_6$, and hence $G_\circ$ is $\Cyc_3$, $\Cyc_6$ or $\Sym_3$. As we are in the setting of the exact sequence (\ref{eq: dihedral 1}), we must have $G\simeq\Dih_6$. Next, if $S'$ is a del Pezzo surface of degree 5, then Lemma~\ref{lemma: dP5 minimal groups} and the exact sequence (\ref{eq: dihedral 1}) show that we must have $m=1$, $G_\circ\simeq\Dih_5$ and hence $G\simeq\GA_1(\FF_5)$. 
\end{proof}

In fact, all possibilities described in Proposition~\ref{prop: dihedral groups on quadric}\ref{Dih 2} do occur. For $\Dih_6$-links to the del Pezzo surface of degree 6, this was already mentioned at the end of Section~\ref{sec: dP6}; see also \cite{IskovskikhNonConjugate}. To construct $\GA_1(\FF_5)$-links to the del Pezzo surface of degree 5, recall that $\GA_1(\FF_5)$ has presentation $\langle \alpha,\beta\ |\ \alpha^5=\beta^4=\id,\ \beta\alpha\beta^{-1}=\alpha^3\rangle$. Let $\alpha\in\Aut(S')\simeq\Sym_5$ be an automorphism of order 5. It is easy to show, see \cite[Lemma 4.16]{YasinskyOdd}, using the Lefschetz fixed-point formula, that $\alpha$ has exactly two fixed points, say $p$ and $q$, in general position on~$S'$. Then $\alpha\beta\cdot p=\alpha^{-2}\beta\alpha\cdot p=\alpha^{-2}\beta\cdot p$, so $\beta\cdot p$ is a fixed point of $\alpha^3$ and hence of $\alpha$; \textit{i.e.}~$\beta\cdot p\in\{p,q\}$. The set $\{p,q\}$ is an orbit of $G$ in general position, and one can associate a link to it; see \cite[Theorem 1.1 and Lemma 4.2]{WolterEquivariantBirationalDP5} for more details.  

Next, let us provide an example of a link to a conic bundle and to $\PP^2$.

\begin{example}\label{ex: dP8 non-rigid}	
	Assume that $m=1$ and that $G_\circ=\Dih_n$ acts on $S=\PP^1\times\PP^1$ ``diagonally'', \textit{i.e.}~by $([x:y],$ $[z:t])\mapsto (A[x:y],A[z:t] )$, where $A\in\PGL_2(\kk)$ are elements of $\Dih_n$. Let us choose the coordinates so that the action of $\Dih_n=\langle r,s\colon r^n=s^2=\id,srs=r^{-1}\rangle$ on each factor is given by $r\colon [x:y]\mapsto [x:\omega y]$, $s\colon [x:y]\mapsto [y:x]$, where $\omega$ denotes a primitive $\supth{n}$ root of unity. Consider two automorphisms of $S$
	\begin{gather*}
		\tau_1\colon ([x:y],[z:t])\longmapsto ([z:t],[x:y]),\\
		\tau_2\colon ([x:y],[z:t])\longmapsto ([t:z],[y:x]).
	\end{gather*}
	Then $\tau_1$ commutes with $G_\circ$, while $\tau_2$ defines a semidirect product $G_\circ\rtimes\langle\tau_2\rangle$, where $\tau_2$ acts by the inversion of $r$ and preserves $s$. Note that $r$ has exactly two fixed points $[1:0]$ and $[0:1]$, which are permuted by $s$. Let
	\[
	p_1=([1:0],[1:0]),\quad p_2=([1:0],[0:1]),\quad p_3=([0:1],[1:0]),\quad p_4=([0:1],[0:1]).
	\]
	The sets $\Omega_1=\{p_1,p_4\}$ and $\Omega_2=\{p_2,p_3\}$ are invariant with respect to $G_\circ$, $\tau_1$ and $\tau_2$, and provide the orbits (in general position) for  $G_1=G_\circ\times\langle\tau_1\rangle$ and $G_2=G_\circ\rtimes\langle\tau_2\rangle$. Blowing up these orbits gives a $G$-conic bundle $S'\to\PP^1$ with two singular fibres (in fact, this is a del Pezzo surface of degree~6; see \cite[Theorem 5]{Iskovskikh_minimal}).  	
	
	Similarly, one can construct examples of links to $\PP^2$. Note that if such a  link exists, Lemma~\ref{lemma: fixed point} implies that $G_\circ$ is an abelian group generated by at most two elements. In particular, $d\leqslant 2$ in (\ref{eq: dihedral 1}). If $d=1$, then $|G_\circ|=2m^2$ and hence $G_\circ$ is an index 2 subgroup of $\Dih_m\times\Dih_m$, which is impossible. If $d=2$, then $G_\circ$ is an index 4 subgroup therein and hence coincides with $\Cyc_{2m}\times\Cyc_{2m}\subset\Dih_{2m}\times\Dih_{2m}$. Making this group act on $\PP^1\times\PP^1$ diagonally, as above, and taking a direct product with $\tau_1$, we get a linearizable action.
\end{example}

Finally, if $n=2$, then $G_\circ=\Klein_4\times_D\Klein_4\subset\Klein_4\times\Klein_4\simeq\Cyc_2^4$ and hence one has the following possibilities for $G_\circ$:
\[
\Klein_4\times\Klein_4\ (D=\id),\quad(\Cyc_2\times\Cyc_2)\times\Cyc_2\simeq\Cyc_2^3\ (D=\Cyc_2),\quad \Klein_4\ (D=\Klein_4).
\]
When $G_\circ\simeq\Klein_4\times\Klein_4\simeq\Cyc_2^4$, the same arguments as in the proof of Proposition~\ref{prop: dihedral groups on quadric}\ref{Dih 1} show that $S$ is $G$-birationally rigid. In the remaining two cases $D\simeq\Cyc_2$ and $D\simeq\Klein_4$, one can construct examples similar to Example~\ref{ex: dP8 non-rigid}. 

\subsection{Case $\boldsymbol{F=\Alt_4}$}\label{subsec: dP8 A4} If $K=\Alt_4$, then $G_\circ=\Alt_4\times\Alt_4$. If $K=\Klein_4$, then we have a short exact sequence
\begin{equation}\label{eq: A4}
1\lra \Klein_4\times\Klein_4\lra G_\circ\lra \Cyc_3\lra 1,
\end{equation}
while for $K=\id$ we simply get $G_\circ=\Alt_4$. Therefore, $G$ is of the form
\begin{equation}\label{eq: A4 extensions}
	(\Alt_4\times\Alt_4)_\bullet\Cyc_2,\quad  ((\Klein_4\times\Klein_4)_\bullet\Cyc_3)_\bullet\Cyc_2\quad \text{or}\quad {\Alt_4}_\bullet\Cyc_2.
\end{equation}
Note that none of the groups $G_\circ$ admits a subgroup $\Cyc_n\times\Cyc_m$ of index at most 2; hence there are no Sarkisov $G$-links of type I on $S$ and no $G$-links leading to $\PP^2$ by~Lemma~\ref{lemma: structure of Gcirc}. Clearly, none of the extensions (\ref{eq: A4 extensions}) is isomorphic to a group from the list~(\ref{eq: dP8 groups}). Since birational Geiser and Bertini involutions and links centred at points of degree 4 lead to a $G$-isomorphic surface, we get that $S$ is $G$-birationally rigid.

Although this is not necessary for our further purposes, let us proceed to explore the existence of $G$-links of type II at points of degree 4. If $G\simeq(\Alt_4\times\Alt_4)_\bullet\Cyc_2$, then there are no such $G$-links as $|G|$ is divisible by 9 and hence $G$ does not embed into automorphism groups of del Pezzo surfaces of degree 4. If such a $G$-link existed for $G\simeq ((\Cyc_2^4)_\bullet\Cyc_3)_\bullet\Cyc_2$, then $G$ would embed into $\Aut(T)$, where $T$ is a quartic del Pezzo surface, and moreover $G=\Aut(T)\simeq\Cyc_2^4\rtimes\Sym_3$. However, one has $\Pic(T)^G\simeq\ZZ$ in this case \cite[Theorem 6.9]{DolgachevIskovskikh}, which gives a contradiction. 

Finally, if $K=\id$, the group $G_\circ=\Alt_4\times_{\Alt_4}\Alt_4\simeq\Alt_4$ acts on $S$ by 
\[
([x:y],[z:t])\longmapsto (g[x:y],\varphi(g)[z:t]),\quad g\in\Alt_4,
\] 
where $\varphi\in\Aut(\Alt_4)$ is a fixed automorphism. The extension ${\Alt_4}_\bullet\Cyc_2$ always splits, and one has $G\simeq\Alt_4\rtimes_\psi\langle\tau\rangle$, where $\tau\in G\setminus G_\circ$. The latter automorphism is of the form
\[
\tau\colon ([x:y],[z:t])\longmapsto (A[z:t],B[x:y]),
\]
where $A,B\in\PGL_2(\kk)$. Since $\tau^2=\id$, we find that $B=A^{-1}$. Let us choose the coordinates on the first factor of $\PP^1\times\PP^1$ so that the derived subgroup $\Klein_4$ of $\Alt_4$ is generated by $[x:y]\mapsto [x:-y]$, $[x:y]\mapsto [y:x]$. A direct computation then shows that an element of order 3 is represented by one of the following matrices: 
\begin{equation}\label{eq: matrices order 3}
\begin{pmatrix}
	i & i\\
	1 & -1
\end{pmatrix},\quad 
\begin{pmatrix}
	i & -i\\
	1 & 1
\end{pmatrix},\quad 
\begin{pmatrix}
	-i & i\\
	1 & 1
\end{pmatrix}, \quad 
\begin{pmatrix}
	-i & -i\\
	1 & -1
\end{pmatrix}.
\end{equation}

Suppose that $\psi=\id$, so $G\simeq\Alt_4\times\langle\tau\rangle$. Below we give an example of $G$-actions which do not give  rise to a Sarkisov link. Their systematic study will be provided elsewhere.

\begin{example}
	Assume $\varphi=\id$, so the group $G_\circ$ acts on $\PP^1\times\PP^1$ ``diagonally'' by $([x:y],[z:t])\mapsto (g\cdot[x:y],$ $g\cdot [z:t] )$, where $g\in \PGL_2(\kk)$. Further, assume that $\tau$ is given by  
	\begin{equation}\label{eq: action of tau}
	([x:y],[z:t])\longmapsto ([z:t],[x:y]).
	\end{equation}
	Obviously, it commutes with $G_\circ$. If $\Omega=G\cdot p$ is an orbit of cardinality $4$ on $S$, then the stabilizer of $p$ is a cyclic group $\langle \tau\rangle\times\langle \delta\rangle\simeq\Cyc_6$, where $\delta\in G_\circ$ is an element of order 3. As the fixed locus of $\tau$ is the diagonal $\Delta\subset\PP^1\times\PP^1$, we have $p\in\Delta$. But $G_\circ$ preserves $\Delta$; hence we have $\Omega\subset\Delta$. So, the orbit $\Omega$ is not in general position on $S$, and hence there is no Sarkisov $G$-link starting from $\Omega$. 

	Now choose a non-identity automorphism $\varphi$; \textit{e.g.} assume that $\varphi\in\Aut(\Alt_4)$ permutes the elements of $\Klein_4\subset\Alt_4$ such that $\Klein_4=\langle\alpha\rangle\times\langle\beta\rangle$ acts by
	\[
	\alpha\colon([x:y],[z:t])\longmapsto ([x:-y],[t:z]),\quad \beta\colon([x:y],[z:t])\longmapsto ([y:x],[t:-z]). 
	\]
	Set $M=\begin{psmallmatrix}
		i & i\\
		1 & -1
	\end{psmallmatrix}$, and define an element of order 3 of $\Alt_4$ and $\tau$ as
	\[
	\gamma\colon\left( 
	\begin{bmatrix}
		x\\y
	\end{bmatrix}, 
	\begin{bmatrix}
	z\\t
	\end{bmatrix} 
	\right)\longmapsto 
	\left( 
	M\begin{bmatrix}
		x\\y
	\end{bmatrix}, 
	M\begin{bmatrix}
		z\\t
	\end{bmatrix} 
	\right),\quad 
	\tau\colon\left( 
	\begin{bmatrix}
		x\\y
	\end{bmatrix}, 
	\begin{bmatrix}
		z\\t
	\end{bmatrix} 
	\right)\longmapsto 
	\left( 
	M\begin{bmatrix}
		z\\t
	\end{bmatrix}, 
	M^{-1}\begin{bmatrix}
		x\\y
	\end{bmatrix} 
	\right),
	\]
	respectively. One easily checks that $\alpha,\beta,\gamma$ generate the group $\langle\alpha,\beta,\gamma\colon\alpha^2=\beta^2=\gamma^3=\id,\gamma\alpha\gamma^{-1}=\alpha\beta=\beta\alpha,\gamma\beta\gamma^{-1}=\alpha\rangle\simeq\Alt_4$ and $\tau$ commutes with this group. However, the automorphism \[\left( 
	\begin{bmatrix}
		x\\y
	\end{bmatrix}, 
	\begin{bmatrix}
		z\\t
	\end{bmatrix} 
	\right)\longmapsto 
	\left( 
	\begin{bmatrix}
		x\\y
	\end{bmatrix}, 
	M\begin{bmatrix}
		z\\t
	\end{bmatrix} 
	\right)\] of $S$ conjugates the group $\Alt_4\times\langle\tau\rangle$ to the one with ``diagonal'' action of $\Alt_4$ and $\tau$ acting as in (\ref{eq: action of tau}).
\end{example}

\subsection{Case $\boldsymbol{F=\Sym_4}$}\label{subsec: dP8 S4} 
If $K=\Sym_4$, then $G_\circ=\Sym_4\times\Sym_4$. If $K=\id$, then $G_\circ=\Sym_4$. In the remaining two cases $K=\Alt_4$ and $K=\Klein_4$, the group $G_\circ$ fits into the short exact sequence 
\begin{equation}\label{eq: S4 one}
1\lra \Alt_4\times\Alt_4\lra G_\circ\lra \Cyc_2\lra 1
\end{equation}
or into the short exact sequence
\begin{equation}\label{eq: S4 two}
1\lra \Klein_4\times\Klein_4\lra G_\circ\lra \Sym_3\lra 1. 
\end{equation}
By Lemma~\ref{lemma: structure of Gcirc}, there are no Sarkisov $G$-links to $\PP^2$ and no Sarkisov $G$-links of type I starting from $S$. Furthermore, none of the extensions
\begin{equation}\label{eq: S4 extensions}
	(\Sym_4\times\Sym_4)_\bullet\Cyc_2,\quad ((\Alt_4\times\Alt_4)_\bullet\Cyc_2)_\bullet\Cyc_2,\quad ((\Klein_4\times\Klein_4)_\bullet\Sym_3)_\bullet\Cyc_2,\quad {\Sym_4}_\bullet\Cyc_2
\end{equation}
of $\Cyc_2$ by $G_\circ$ is isomorphic to a group from the list (\ref{eq: dP8 groups}). Note that there are no Sarkisov $G$-links of type II starting at a point of degree 4. Indeed, otherwise, $G$ has a subgroup of index 4 which fixes a point on $S$, and hence $G_\circ$ has an abelian subgroup of the form $\Cyc_n\times\Cyc_m$ and of index $2$ or $4$, which is clearly not the case (alternatively, one can again argue using Remark~\ref{rem: dP4 auto}). We conclude that $S$ is $G$-birationally rigid.

\subsection{Case $\boldsymbol{F=\Alt_5}$}\label{subsec: dP8 A5} Since $\Alt_5$ is simple, one has $K=\id$ or $K=\Alt_5$, so either $G\simeq(\Alt_5\times_{\Alt_5}\Alt_5)_\bullet\Cyc_2\simeq{\Alt_5}_\bullet\Cyc_2$ or $G\simeq(\Alt_5\times\Alt_5)_\bullet\Cyc_2$. Clearly, none of these groups is isomorphic to a group from (\ref{eq: dP8 groups}), unless in the first extension we get $G\simeq\Sym_5$. However, by Lemma~\ref{lemma: fixed point}, there exists\footnote{Alternatively, if such link existed, then the surface $T$ in the diagram (\ref{eq: Sarkisov II}) is a cubic surface with an action of $\Sym_5$. It is well known that $T$ must be the Clebsch diagonal cubic. One always has $\Pic(T)^{\Sym_5}\simeq\ZZ$; see \cite[Theorem 6.14]{DolgachevIskovskikh}.} no Sarkisov $\Sym_5$-link $\chi\colon S\dashrightarrow S'$ to a del Pezzo surface $S'$ of degree~5, as the centre of $(\eta')^{-1}$ in (\ref{eq: Sarkisov II}) would be a point of degree 2. Further, none of the groups $G_\circ=\Alt_5$ or $G_\circ=\Alt_5\times\Alt_5$ has a subgroup of the form $\Cyc_n\times\Cyc_m$ of index at most 2 in $G_\circ$, so there are no $G$-links of type~I and no $G$-links to $\PP^2$ by Lemma~\ref{lemma: structure of Gcirc}. By Remark~\ref{rem: dP4 auto}, $S$ does not admit $G$-links centred at a point of degree~$4$. Similarly, $G$ does not embed into automorphism group of del Pezzo surfaces of degree 1 or 2; see \cite[Sections~8.7 and~8.8]{DolgachevClassicalAG}. We conclude that $S$ is $G$-birationally superrigid.

\begin{remark}\label{rem: proof bis} The results of previous sections allow us to complete the proof of the main theorem in an alternative way. Indeed, $S\simeq\PP^1\times\PP^1$ may fail to be $G$-birationally rigid exactly in the following cases:
\begin{a-enumerate}
	\item\label{r:pb-a} $G=(\Cyc_n\times_D\Cyc_n)_\bullet\Cyc_2$ is a group from Section~\ref{subsec: dP8 Cyc} and $G_\circ=\Cyc_n\times_D\Cyc_n$. 
	\item\label{r:pb-b}  $G=(\Dih_n\times_D\Dih_n)_\bullet\Cyc_2$ is a group from Section~\ref{subsec: dP8 Dih}, where $G_\circ$ is as in the case~\ref{sec: series 1}.
	\item\label{r:pb-c}  $G=(\Cyc_2^3)_\bullet\Cyc_2$ or $G={\Klein_4}_\bullet\Cyc_2$ are groups from Section~\ref{subsec: dP8 Dih}.
\end{a-enumerate}

The groups $G_\circ$ from~\ref{r:pb-a}  are abelian and generated by at most two elements, hence cannot contain any of the groups $G_\circ$ of Sections~\ref{subsec: dP8 Dih}--\ref{subsec: dP8 A5} unless $G_\circ=\Klein_4$ is as in~\ref{r:pb-c}. However, in the latter case, $S$ is not $G$-birationally rigid. The groups $G_\circ$ from Sections~\ref{subsec: dP8 Dih}--\ref{subsec: dP8 A5} cannot embed into the groups $G_\circ$ from~\ref{r:pb-c} by  order reasons. 

Suppose that $G_\circ$ is a group from~\ref{r:pb-b}. Then there exists a Sarkisov $G$-link $\chi\colon S\dashrightarrow S'$ leading to a different $G$-Mori fibre space $S'\to Z$, where $Z=\pt$ or $Z=\PP^1$; all possibilities are given by Proposition~\ref{prop: dihedral groups on quadric}. If $S'=\PP^2$, then $G$ fixes a point on $S$, and hence the same is true for any subgroup $H\subset G$, showing that $H$ is linearizable as well (see Remark~\ref{rem: cyclic group quadric}). If $S'\to\PP^1$ is a $G$-conic bundle, then $G$ has an orbit of size 2 on $S$, and hence $H$ has an orbit of size 1 or 2; the result follows. If $S'$ is a $G$-del Pezzo surface of degree 6, then $G\simeq \Dih_6$ and $G_\circ\simeq\Dih_3$. Then $H_\circ\in\{\id,\Cyc_2,\Cyc_3\}$ and $S$ is not $H$-birationally rigid. Finally, if $S'$ is a del Pezzo surface of degree 5, then $G\simeq\GA_1(\FF_5)$ and $G_\circ\simeq\Dih_5$; hence, $H_\circ\in\{\id,\Cyc_2,\Cyc_5\}$ and once again $S$ fails to be $H$-birationally rigid.

\end{remark}

\section{The projective space of dimension 3 and the ``mixed'' arithmetico-geometric case}

\subsection{$\boldsymbol{G}$-birational rigidity of $\boldsymbol{\PP^3}$}\label{sec: projective space} Inspired by Blichfeldt's classification \cite{Blichfeldt} of finite subgroups of $\PGL_4(\CC)$, I.~Cheltsov and C.~Shramov managed to describe all the cases in which $\PP^3_\CC$ is $G$-birationally rigid. 

\begin{theorem}[\textit{cf.} {\cite[Theorem 1.1]{CheltsovShramovFiniteCollineationGroups}}]
	The 3-dimensional complex projective space $\PP^3$ is $G$-birationally rigid if and only if $G$ is primitive and not isomorphic to $\Alt_5$ and $\Sym_5$.
\end{theorem}

This immediately implies a positive answer to Cheltsov--Koll\'{a}r's question for $X=\PP^3$. 

\begin{corollary}
	Let $G$ be a finite group and $H\subset G$ be its subgroup. If\, $\PP^3$ is $H$-birationally rigid, then $\PP^3$ is $G$-birationally rigid.
\end{corollary}
\begin{proof}
	Recall (see Definition~\ref{def: Blichfeldt}) that finite subgroups of $\Aut(\PP^3)\simeq\PGL_4(\CC)$ are either transitive (\textit{i.e.}~do not fix any point and do not leave any line invariant) or intransitive. Transitive groups are either imprimitive (\textit{i.e.}~leave a union of two skew lines invariant or have an orbit of length 4) or primitive.
	
	Now, if $\PP^3$ is $H$-birationally rigid, then $H$ is transitive, and hence $G$ is transitive as well. Clearly, if $G$ leaves a union of two skew lines invariant, then the same is true for $H$. If $G$ has an orbit of length 4 and $H$ has no orbit of length 4, then $H$ either fixes a point, or permutes two points in $\PP^3$ and hence has an invariant line. Both cases are not possible as $H$ is transitive. So, $G$ is primitive. It remains to notice that $G$ is not isomorphic to $\Alt_5$ or $\Sym_5$ as all proper subgroups of $\Sym_5$, not isomorphic to $\Alt_5$, are not primitive; see \textit{e.g.} \cite[Appendix]{CheltsovShramovFiniteCollineationGroups}.
\end{proof}

\subsection{Cheltsov--Koll\'{a}r's question in the arithmetico-geometric case}\label{subsec: mixed}

Assume that the base field $\kk$ is not algebraically closed. A natural generalization of Cheltsov--Koll\'{a}r's question to the ``mixed'' setting of Remark~\ref{rem: mixed case} would be: if $X$ is an $H$-birationally rigid $H$-Fano variety over $\overline{\kk}$, then must $X$ be $G$-birationally rigid over $\kk$? We claim that the answer to this question is \emph{negative} already for $H=G$. 

Let $p\equiv 1$ (mod $3$) be a prime number, fix a non-trivial homomorphism $\Cyc_3\to\Aut(\Cyc_p)$, and consider the corresponding semidirect product $G=\Cyc_p\rtimes\Cyc_3$. By the result of C.~Shramov \cite[Theorem 1.3(ii)]{Shramov2}, there exist a field $\kk$ (of characteristic zero) and a non-trivial Severi--Brauer surface $S$ over $\kk$ such that $G\hookrightarrow\Aut(S)$. Any Sarkisov $G$-link $\chi\colon S\dashrightarrow S'$ centred at a point of degree 3 leads to the {\it opposite} Severi--Brauer surface $S'=S^{\rm op}$; if $S$ corresponds to a central simple $\kk$-algebra $A$, then by definition, $S^{\rm op}$ is the unique Severi--Brauer surface corresponding to $A^{\rm op}$, the inverse of $A$ in the Brauer group $\Br(\kk)$. Since $S^{\rm op}$ is never isomorphic to $S$, we conclude that $S$ is not $G$-birationally rigid. However, passing to the algebraic closure of $\kk$, one has $S_{\overline{\kk}}\simeq\PP_{\overline{\kk}}^2$, which is $G$-birationally rigid by Sakovics' Theorem~\ref{thm: Sakovics}. For more recent results on birational and biregular self-maps of Severi--Brauer surfaces, see \cite{Shramov1,TrepalinSB,BlancSchneiderYasinsky}.

\newcommand{\etalchar}[1]{$^{#1}$}
\providecommand{\bysame}{\leavevmode\hbox to3em{\hrulefill}\thinspace}


\begin{thebibliography}{CMYZ22+++}

\bibitem[AO18]{AbbanOkadaPfaffian}
H.~Ahmadinezhad and T.~Okada, \emph{Birationally rigid {Pfaffian} {Fano} 3-folds},  Algebr.\ Geom.\ \textbf{5} (2018), no.~2, 160--199,
  \doi{10.14231/AG-2018-006}.

\bibitem[Avi20]{AvilovFormsSegre}
A.\,A.~Avilov, \emph{Forms of the {Segre} cubic}, Math.\ Notes \textbf{107}
  (2020), no.~1, 3--9, \doi{10.1134/S0001434620010010}.

\bibitem[Bea07]{BeauvilleElementary}
A.~Beauville, \emph{{{\(p\)}}-elementary subgroups of the {Cremona} group}, J.\
  Algebra \textbf{314} (2007), no.~2, 553--564,
  \doi{10.1016/j.jalgebra.2005.07.040}.

\bibitem[BB73]{BialynickiBirula}
A.~Bialynicki-Birula, \emph{Some theorems on actions of algebraic groups}, Ann.\
  Math.~(2) \textbf{98} (1973), no.~3, 480--497, \doi{10.2307/1970915}.

\bibitem[BLZ21]{BLZ}
J.~{Blanc}, S.~{Lamy} and S.~{Zimmermann}, \emph{{Quotients of higher
  dimensional Cremona groups}}, Acta Math.\ \textbf{226} (2021), no.~2, 211--318, \doi{10.4310/acta.2021.v226.n2.a1}.

\bibitem[BSY22]{BlancSchneiderYasinsky}
J.~Blanc, J.~Schneider and E.~Yasinsky, \emph{Birational maps of
  {Severi-Brauer} surfaces, with applications to {C}remona groups of higher
  rank}, preprint \arXiv{2211.17123} (2022).
  
\bibitem[Bli17]{Blichfeldt}
H.\,F.~Blichfeldt, \emph{Finite Collineation Groups}, University of Chicago
  Press, 1917.

\bibitem[Che08]{CheltsovLCThresholds}
I.~Cheltsov, \emph{Log canonical thresholds of del {Pezzo} surfaces}, Geom.\
  Funct.\ Anal.\ \textbf{18} (2008), no.~4, 1118--1144,
  \doi{10.1007/s00039-008-0687-2}.

\bibitem[CMYZ22]{CMYZ}
I.~Cheltsov, F.~Mangolte, E.~Yasinsky and S.~Zimmermann, \emph{Birational
  involutions of the real projective plane}, preprint \arXiv{2208.00217}  (2022).

\bibitem[CS22]{CheltsovSarikyan}
I.~Cheltsov and A.~Sarikyan, \emph{Equivariant pliability of the projective
  space},  preprint \arXiv{2202.09319} (2022).
  
\bibitem[CS16]{CheltsovShramovIcosahedron}
I.~Cheltsov and C.~Shramov, \emph{Cremona groups and the icosahedron}, Monogr.\
  Res.\ Notes Math., CRC Press, Boca Raton, FL, 2016, \doi{10.1201/b18980}.

\bibitem[CS19]{CheltsovShramovFiniteCollineationGroups}
\bysame, \emph{Finite collineation groups and birational rigidity}, Sel.\ Math.,
New Ser.~\textbf{25} (2019), no.~5, Paper No.~71,  \doi{10.1007/s00029-019-0516-5}.

\bibitem[Cor00]{CortiSingularitiesOfLinearSystems}
A.~Corti, \emph{Singularities of linear systems and 3-fold birational
  geometry}, In: \emph{Explicit birational geometry of 3-folds}, pp.~259--312, Cambridge  Univ.\ Press, Cambridge, 2000, 
\doi{10.1017/CBO9780511758942.007}.

\bibitem[Dol12]{DolgachevClassicalAG}
I.\,V.~Dolgachev, \emph{Classical algebraic geometry. {A} modern view},
Cambridge Univ.\ Press, Cambridge, 2012.

\bibitem[DD16]{DolgachevDuncanFixed}
I.~Dolgachev and A.~Duncan, \emph{Fixed points of a finite subgroup of the
  plane {Cremona} group}, Algebr.\ Geom.\ \textbf{3} (2016), no.~4, 441--460, \doi{10.14231/AG-2016-021}.

\bibitem[DI09a]{DolgachevIskovskikh}
I.\,V.~Dolgachev and V.\,A.~Iskovskikh, \emph{Finite subgroups of the plane
{Cremona} group}, In: \emph{Algebra, arithmetic, and geometry. In honor of Yu.\,I.~Manin. Vol. I}, pp.~443--548, Progr.\ Math.\ vol.~269, Birkh\"auser, Boston, MA, 2009, \doi{10.1007/978-0-8176-4745-2_11}.

\bibitem[DI09b]{DolgachevIskovskikhPerfectField}
\bysame, \emph{On elements of prime order in the plane {Cremona} group over a
  perfect field}, Int.\ Math.\ Res.\ Not.\ (2009), no.~18, 3467--3485,
  \doi{10.1093/imrp/rnp061}.

\bibitem[dDM19]{MauriDasDores}
L.~das Dores and M.~Mauri, \emph{{{\(G\)}}-birational superrigidity of {Del}
  {Pezzo} surfaces of degree 2 and 3}, Eur.~J.\ Math.\ \textbf{5} (2019), no.~3,
  798--827, \doi{10.1007/s40879-018-0298-x}.
  
\bibitem[Flo20]{Floris}
E.~Floris, \emph{A note on the {{\(G\)}}-{Sarkisov} program}, Enseign.\ Math.\
  (2) \textbf{66} (2020), no.~1-2, 83--92,
  \doi{10.4171/LEM/66-1/2-5}.

\bibitem[FZ24]{FlorisZikas}
E.~Floris and S.~Zikas, \emph{Umemura quadric fibrations and maximal subgroups
  of {$Cr_n(\mathbb{C})$}}, preprint \arXiv{2402.05021} (2024).

\bibitem[Gou89]{Goursat}
E.~Goursat, \emph{Sur les substitutions orthogonales et les divisions
  r{\'e}guli{\`e}res de l'espace}, Ann.\ Sci.\ {\'E}c.\ Norm.\ Sup{\'e}r.~(3)
  \textbf{6} (1889), 9--102, 1889, \doi{10.24033/asens.317}.

\bibitem[Hos96]{HosohQuarticDP}
T.~Hosoh, \emph{Automorphism groups of quartic del {Pezzo} surfaces}, J.~Algebra \textbf{185} (1996), no.~2, 374--389,  \doi{10.1006/jabr.1996.0331}.

\bibitem[Isk80]{Iskovskikh_minimal}
V.\,A.~Iskovskikh, \emph{Minimal models of rational surfaces over arbitrary
  fields}, Math.\ USSR, Izv.\ \textbf{14} (1980), no.~1, 17--39,
  \doi{10.1070/IM1980v014n01ABEH001064}.

\bibitem[Isk96]{Isk1996}
\bysame, \emph{{Factorization of birational maps of rational surfaces
  from the viewpoint of Mori theory}}, Russ.\ Math.\ Surv.\ \textbf{51} (1996),
no.~4, 585--652,
\doi{10.1070/RM1996v051n04ABEH002962}.

\bibitem[Isk08]{IskovskikhNonConjugate}
\bysame, \emph{Two non-conjugate embeddings of {{\(S_3\times\mathbb
  Z_2\)}} into the {Cremona} group. {II}}, In: \emph{Algebraic geometry in East
  Asia---Hanoi 2005}, pp.~251--267, Adv.\ Stud.\ Pure Math., vol.~50, Math.\ Soc.\ Japan, Tokyo, 2008,
\doi{10.2969/aspm/05010251}.


\bibitem[Kle19]{KleinIcosahedron}
F.~Klein, \emph{Lectures on the icosahedron and the solution of equations of
  the fifth degree} (translated by G.\,G.~Morrice, with a new
introduction and commentaries by P.~Slodowy, translated by L.~Yang, reprint of the English translation of the 1884 German original edition), CTM.\ Class.\ Top.\ Math., vol.~5, Higher Education Press, Beijing, 2019.

\bibitem[Kol09]{KollarRigidity}
J.~Koll{\'a}r, \emph{Birational rigidity of {Fano} varieties and field
  extensions}, Proc.\ Steklov Inst.\ Math.\ \textbf{264} (2009), no.~1, 96--101, 
  \doi{10.1134/S008154380901012X}.

\bibitem[LS21]{LamySchneider}
S.~{Lamy} and J.~{Schneider}, \emph{Generating the plane Cremona groups by
involutions},  Algebr.\ Geom.\ \textbf{11} (2024), no.~1, 111-–162,
\doi{10.14231/ag-2024-004}.

\bibitem[LZ20]{LamyZimmermann}
S.~Lamy and S.~Zimmermann, \emph{Signature morphisms from the {Cremona} group
  over a non-closed field}, J.~Eur.\ Math.\ Soc.\ (JEMS) \textbf{22} (2020),
  no.~10, 3133--3173, \doi{10.4171/JEMS/983}.

\bibitem[Man66]{ManinRationalI}
Ju.\,I.~Manin, \emph{Rational surfaces over perfect fields}, Publ.\ Math.\ Inst.\  Hautes {\'E}tud.\ Sci.\ \textbf{30} (1966), 55--97 (Russian), 
  \doi{10.1007/BF02684356}.
  
\bibitem[Pin24]{Pinardin}
A.~Pinardin, \emph{{{\(G\)}}-solid rational surfaces}, Eur.~J.\ Math.\
  \textbf{10} (2024), no.~2, Paper No.~33, \doi{10.1007/s40879-024-00747-z}.

\bibitem[Pro21]{ProkhorovGMMP}
Y.\,G.~Prokhorov, \emph{Equivariant minimal model program}, Russ.\ Math.\ Surv.\
  \textbf{76} (2021), no.~3, 461--542, \doi{10.1070/RM9990}.

\bibitem[RZ18]{RobayoZimmermann}
M.\,F.~Robayo and S.~Zimmermann, \emph{Infinite algebraic subgroups of the real
  {Cremona} group}, Osaka J.~Math.\ \textbf{55} (2018), no.~4, 681--712.

\bibitem[Sak19]{SakovicRigidity}
D.~Sakovics, \emph{{{\(G\)}}-birational rigidity of the projective plane}, Eur.\
  J.~Math.\ \textbf{5} (2019), no.~3, 1090--1105,
  \doi{10.1007/s40879-018-0261-x}.

\bibitem[SZ21]{SchneiderZimmermann}
J.~Schneider and S.~Zimmermann, \emph{Algebraic subgroups of the plane
  {Cremona} group over a perfect field}, {\'E}pijournal de G{\'e}om.\
  Alg{\'e}br.\ \textbf{5} (2021), Art.~14,
  \doi{10.46298/epiga.2021.6715}.
  
\bibitem[Shr20]{Shramov1}
C.\,A. Shramov, \emph{Birational automorphisms of {Severi}-{Brauer} surfaces},
  Sb.\ Math.\ \textbf{211} (2020), no.~3, 466--480,
  \doi{10.1070/SM9304}.

\bibitem[Shr21]{Shramov2}
\bysame, \emph{Finite groups acting on {Severi}-{Brauer} surfaces}, Eur.~J.\
  Math.\ \textbf{7} (2021), no.~2, 591--612,
  \doi{10.1007/s40879-020-00448-3}.

\bibitem[Tre16]{TrepalinConicBundles}
A.~Trepalin, \emph{Quotients of conic bundles}, Transform.\ Groups \textbf{21}
  (2016), no.~1, 275--295,
  \doi{10.1007/s00031-015-9342-9}.

\bibitem[Tre18]{TrepalinDPhigh}
  \bysame, \emph{Quotients of del {Pezzo} surfaces of high degree}, Trans.\ Amer.\  Math.\ Soc.\ \textbf{370} (2018), no.~9, 6097--6124,
  \doi{10.1090/tran/7130}.

\bibitem[Tre19]{TrepalinDelPezzo}
\bysame, \emph{Quotients of del {Pezzo} surfaces}, Int.~J.\ Math.\ \textbf{30}
  (2019), no.~12, 1950068, \doi{10.1142/S0129167X1950068X}.

\bibitem[Tre21]{TrepalinSB}
\bysame, \emph{Quotients of {Severi}-{Brauer} surfaces}, Dokl.\ Math.\
  \textbf{104} (2021), no.~3, 390--393,
  \doi{10.1134/S106456242106017X}.

\bibitem[Tre23]{TrepalinDP8}
\bysame, \emph{Birational classification of pointless del {Pezzo} surfaces
  of degree 8}, Eur.~J.\ Math.\ \textbf{9} (2023), no.~1, Paper No.~2, 
  \doi{10.1007/s40879-023-00591-7}.

\bibitem[Wol18]{WolterEquivariantBirationalDP5}
J.~Wolter, \emph{Equivariant birational geometry of quintic del {Pezzo}
  surface}, Eur.~J.\ Math.\ \textbf{4} (2018), no.~3, 1278--1292,
  \doi{10.1007/s40879-018-0272-7}.

\bibitem[Yas16]{YasinskyOdd}
E.~Yasinsky, \emph{Subgroups of odd order in the real plane {Cremona} group},
  J.~Algebra \textbf{461} (2016), 87--120,
  \doi{10.1016/j.jalgebra.2016.04.019}.

\bibitem[Yas22]{YasinskyDelPezzo}
\bysame, \emph{Automorphisms of real del {Pezzo} surfaces and the real plane
  {Cremona} group}, Ann.\ Inst.\ Fourier \textbf{72} (2022), no.~2, 831--899, 
  \doi{10.5802/aif.3460}.

\end{thebibliography}
\end{document}